\documentclass[a4paper, 11pt,reqno]{amsart}
\usepackage[utf8]{inputenc}
\usepackage{amsfonts,extarrows,comment}
\usepackage{bm}
\usepackage{amsmath,amsthm,amssymb,amsfonts,mathrsfs,upgreek}
\usepackage{stmaryrd}
\usepackage{tikz}
\usetikzlibrary{intersections,calc}
\usetikzlibrary{quotes,angles} 
\usepackage{interval}

\usepackage{dsfont}
\usepackage{mathtools}
\usepackage[all]{xy}
\usepackage{dirtytalk}

\usepackage{graphicx}
\usepackage{tikz-cd}
\usepackage{todonotes}
\usepackage{mathtools}
\usepackage{xcolor}

\usepackage{subcaption}
\usepackage[margin=0.8in]{geometry}
\usepackage[colorlinks=true,citecolor=magenta,linkcolor=blue]{hyperref}
\usepackage[justification=centering]{caption}

\newif\ifHideFoot
\HideFoottrue  
\HideFootfalse  

\ifHideFoot
  
\newcommand{\Zhiyuan}[1]{}
\newcommand{\Haoyu}[1]{}
\newcommand{\Yiran}[1]{}

\else 

\newcommand{\marg}[1]{\normalsize{{
			\color{red}\footnote{{\color{blue}#1}}}{\marginpar[\vskip
			-.25cm{\color{red}\hfill$\Rightarrow$\tiny\thefootnote}]{\vskip
				-.2cm{\color{red}$\Leftarrow$\tiny\thefootnote}}}}}

\newcommand{\Zhiyuan}[1]{\marg{(Zhiyuan) #1}}
\newcommand{\Haoyu}[1]{\marg{(Haoyu) #1}}
\newcommand{\Yiran}[1]{\marg{(Yiran) #1}}
\fi

\usepackage{listings}
\lstdefinestyle{Mathematica}{
    language        =   Mathematica, 
    basicstyle      =   \ttfamily,
    numberstyle     =   \ttfamily,
    keywordstyle    =   \color{blue},
    keywordstyle    =   [2] \color{teal},
    stringstyle     =   \color{magenta},
    commentstyle    =   \color{red}\ttfamily,
    breaklines      =   true,   
    columns         =   fixed,  
    basewidth       =   0.5em,
}

\numberwithin{equation}{section}

1
\numberwithin{subsection}{section}

\allowdisplaybreaks[1]
\usepackage{enumitem}
\setlist[enumerate]{label=\rm{(\roman*)},leftmargin=\parindent,itemindent=\parindent,labelsep=5pt}
\newlist{TFAE}{enumerate}{1}
\setlist[TFAE]{label=\rm{(\alph*)},labelindent=2\parindent}

\newlist{enumeratea}{enumerate}{2}
\setlist[enumeratea]{label=\rm{(\emph{\alph*})}}

\newlist{enumerate1}{enumerate}{2}
\setlist[enumerate1]{label=(\emph{\arabic*})}

\newtheorem*{namedtheorem}{\theoremname}
\newcommand{\theoremname}{testing}

\newtheorem{theorem}{Theorem}[section]
\newtheorem{proposition}[theorem]{Proposition}

\newtheorem{corollary}[theorem]{Corollary}
\newtheorem{lemma}[theorem]{Lemma}

\theoremstyle{definition}
\newtheorem{definition}[theorem]{Definition}
\newtheorem{definition/proposition}[theorem]{Definition/Proposition}

\newtheorem{remark}[theorem]{Remark}
\newtheorem*{remark*}{Remark}

\theoremstyle{remark}

\newcommand\cA{\mathcal{A}}

\newcommand\cF{\mathcal{F}}

\newcommand\cI{\mathcal{I}}

\newcommand\cM{\mathcal{M}}

\newcommand\cO{\mathcal{O}}
\newcommand\cP{\mathcal{P}}

\newcommand\cW{\mathcal{W}}

\newcommand\CC{\mathbb{C}}

\newcommand\RR{\mathbb{R}}

\newcommand\ZZ{\mathbb{Z}}

\newcommand\bM{\mathbf{M}}

\newcommand\bP{\mathbf{P}}

\newcommand\bT{\mathbf{T}}

\newcommand\rD{\mathrm{D}}

\newcommand\rH{\mathrm{H}}

\newcommand\rL{\mathrm{L}}

\newcommand\rR{\mathrm{R}}

\newcommand\rmd{\mathrm{d}}

\DeclareMathOperator{\NS}{NS}
\DeclareMathOperator{\Hom}{Hom}

\DeclareMathOperator{\alg}{alg}

\DeclareMathOperator{\ch}{ch}
\DeclareMathOperator{\td}{td}

\DeclareMathOperator{\cok}{cok}
\DeclareMathOperator{\Coh}{Coh}

\DeclareMathOperator{\dCat}{D^b}    
\DeclareMathOperator{\nGgp}{\mathit{K}_{num}}
\DeclareMathOperator{\Halg}{\mathrm{H}_{alg}^{*}}

\DeclareMathOperator{\rk}{rk}
\DeclareMathOperator{\coho}{\rm H}

\DeclareMathOperator{\im}{\mathrm{Im}}
\DeclareMathOperator{\re}{\mathrm{Re}}

\newcommand{\mcolon}{\mathbin{:}}
\newcommand{\defeq}{\vcentcolon=}

\newcommand{\BN}{\mathbf{BN}}

\newcommand{\mybrace}[2]{\left\{ #1 \,\middle|\, #2 \right\}}
\newcommand{\pairing}[2]{\left\langle #1,#2 \right\rangle}

\newcommand{\sab}{\sigma_{\alpha,\beta}}
\newcommand{\tri}{\pmb{\triangle}}

\newcommand{\barsigma}{\bar{\sigma}}

\newcommand{\bigbrace}[1]{ \left\{ #1 \right\} }

\makeatletter
\newcommand*{\da@rightarrow}{\mathchar"0\hexnumber@\symAMSa 4B }
\newcommand*{\da@leftarrow}{\mathchar"0\hexnumber@\symAMSa 4C }
\newcommand*{\xdashrightarrow}[2][]{%
  \mathrel{%
    \mathpalette{\da@xarrow{#1}{#2}{}\da@rightarrow{\,}{}}{}%
  }%
}
\newcommand{\xdashleftarrow}[2][]{%
  \mathrel{%
    \mathpalette{\da@xarrow{#1}{#2}\da@leftarrow{}{}{\,}}{}%
  }%
}
\newcommand*{\da@xarrow}[7]{%
  \sbox0{$\ifx#7\scriptstyle\scriptscriptstyle\else\scriptstyle\fi#5#1#6\m@th$}%
  \sbox2{$\ifx#7\scriptstyle\scriptscriptstyle\else\scriptstyle\fi#5#2#6\m@th$}%
  \sbox4{$#7\dabar@\m@th$}%
  \dimen@=\wd0 %
  \ifdim\wd2 >\dimen@
    \dimen@=\wd2 %
  \fi
  \count@=2 %
  \def\da@bars{\dabar@\dabar@}%
  \@whiledim\count@\wd4<\dimen@\do{%
    \advance\count@\@ne
    \expandafter\def\expandafter\da@bars\expandafter{%
      \da@bars
      \dabar@ 
    }%
  }%
  \mathrel{#3}%
  \mathrel{%
    \mathop{\da@bars}\limits
    \ifx\\#1\\%
    \else
      _{\copy0}%
    \fi
    \ifx\\#2\\%
    \else
      ^{\copy2}%
    \fi
  }%
  \mathrel{#4}%
}
\makeatother

\title{Mukai's program for non-primitive curves on K3 surfaces}

\author{Yiran Cheng}
\address{Shanghai Center for Mathematical Sciences\\
Fudan University\\
2005 Songhu Road\\
20438 Shanghai, China}
\email{yrcheng21@m.fudan.edu.cn}

\author{Zhiyuan Li}
\address{Shanghai Center for Mathematical Sciences\\
Fudan University\\
2005 Songhu Road\\
20438 Shanghai, China}
\email{zhiyuan\_li@fudan.edu.cn}

\author{Haoyu Wu}
\address{Shanghai Center for Mathematical Sciences\\
Fudan University\\
2005 Songhu Road\\
20438 Shanghai, China}
\email{hywu18@fudan.edu.cn}

\begin{document}
\tikzset{%
    add/.style args={#1 and #2}{
        to path={%
 ($(\tikztostart)!-#1!(\tikztotarget)$)--($(\tikztotarget)!-#2!(\tikztostart)$)%
  \tikztonodes},add/.default={.2 and .2}}
}  

\begin{abstract}
Mukai’s program in \cite{Mukai1988} seeks to recover a K3 surface $X$ from any curve $C$ on it by exhibiting it as a Fourier--Mukai partner to a Brill--Noether locus of vector bundles on the curve.  In the case $X$ has Picard number one and the curve $C\in |H|$ is primitive,  this was confirmed by Feyzbakhsh in \cite{Feyzbakhsh2020, Feyzbakhsh2020+} for $g\geq 11$ and $g\neq 12$. More recently, Feyzbakhsh has shown in \cite{Feyzbakhsh2022} that  certain moduli spaces of stable bundles on $X$ are isomorphic to the Brill--Noether locus of curves in $|H|$  if $g$ is sufficiently large. In this paper, we work with irreducible curves in a non-primitive ample linear system $|mH|$ and prove that Mukai's program is valid for any irreducible curve when $g\neq 2$, $mg\geq 11$ and $mg\neq 12$. Furthermore, we introduce the destabilising regions to improve Feyzbakhsh's analysis in \cite{Feyzbakhsh2022}. We show that there are hyper-K\"ahler varieties as Brill--Noether loci of curves in every dimension. 
\end{abstract}
\maketitle

\section{Introduction}
Let $\cF_g$ be the moduli stack of primitively polarised K3 surface $(X,H)$ with $H^2 = 2g-2$ and let  $\cP_{g_m}$ be the moduli stack of triples $[(X,H,C)]$ such that $(X,H)\in \cF_g$ and $C \in |mH|$ a smooth curve of genus $g_m=m^2(g-1)+1$.  There are natural forgetful maps 
\[
\begin{tikzcd}
      & \cP_{g_m} \arrow[ld, "\Psi_{g_m}"'] \arrow[rd,  "\Phi_{g_m}"] &       \\
\cM_{g_m} &                                               & \cF_g
\end{tikzcd}
\]
where the fibre of $\Phi_{g_m}$ over $(X,H)\in \cF_g$ is an open subset of the linear system $|m H|$.  In recent years, there is a series of works studying  the rational map $\Psi_{g_m}$ and its rational inverse. For instance, Mukai has proved  in \cite{Mukai1988} that the rational map $\Psi_{g}$ is dominant if $g\leq 11$ and $g \neq 10$, while  Ciliberto--Lopez--Miranda \cite{Ciliberto1993} showed that  $\Psi_g $ is generically injective if $ g\geq 11$ and $g\neq 12$.  More generally,  due to the results  of  \cite{Ciliberto1998} and the recent work in Ciliberto--Dedieu--Sernesi  \cite{Ciliberto2020, Ciliberto2021}, the map $\Psi_{g_m}$ is generically finite  when $mg\geq 11$ and $mg\neq 12$.   There are  other approaches for the case $m\geq 2,\,g\geq 8$ or $m\geq 5, g=7$ (cf.~\cite{Ciliberto2017, Kemeny2015}). 

On the contrary, Mukai has proposed a program in \cite{Mukai2001}  to find the rational inverse of $\Psi_g$ by relating the K3 surface with the Brill--Noether locus of vector bundles on curves. This has been confirmed by  Mukai in \cite{Mukai1996}  when $g=11$ and later on Arbarello--Bruno--Sernesi \cite{Arbarello2014} generalised his result to the case  $g=4k+3$ for some $k$.
In recent years, Feyzbakhsh  has verified this program in \cite{Feyzbakhsh2020,Feyzbakhsh2020+} for all $g \geq 11$ and $g\neq 12$ by using the Bridgeland stability conditions. In this paper, we would like to investigate the rational inverse of the map $\Psi_{g_m}$ for arbitrary $m\in \ZZ_{>0}$ via Mukai's program for curves in non-primitive classes.

\subsection*{Main results}  Let $(X,H)$ be a primitively polarised K3 surface of genus $g$ with Picard number one. Let $\Halg(X)\cong \ZZ^{\oplus 3}$ be the algebraic Mukai lattice and let $\bM(v)$ be the moduli space of $H$-Gieseker semistable coherent sheaves on $X$ with Mukai vector $v\in \Halg(X)$.  The first main result  of this paper is 

\begin{theorem}\label{mainthm}
Assume $g>2$.   Let $C\in |mH|$ be an irreducible curve. Then if $mg\geq 11$ and $mg\neq 12$,  there exists a primitive Mukai vector $v=(r,c,s)$ with $v^2=0$ such that the restriction map
\begin{align}\label{eq_restriction map}
\begin{split}
    \psi \colon \bM(v) &\rightarrow \mathbf{BN}_C(v) \\
    E &\mapsto E|_C \\
\end{split}
\end{align}
is an isomorphism.  Here,   $\BN_C(v)$  is the Brill--Noether locus of slope stable vector bundles on $C$ with rank $r $, degree $2m c(g-1)$ and  $h^0\geq r+s$.
\end{theorem}
As in \cite{Feyzbakhsh2020}, one can then reconstruct $X$ as the moduli space of twisted sheaves on $\BN_C(v)$.  Clearly, such reconstruction is unique for K3 surfaces in $\cF_g$ of Picard number one. Due to the results of  \cite{Dutta2021},  when $m>1$, generic curves  in $|m H|$  have maximal variation,  i.e. the rational map  
\begin{equation}
|m H| \xdashrightarrow{\text{$\Psi_{g_m}$}} \cM_{g_m}
\end{equation}
is generically finite.  One can also deduce the generic quasi-finiteness of $\Psi_{g_m}$ from Theorem \ref{mainthm} when $m>1, g\geq 3$, $mg\geq 11$ and $mg\neq 12$.  When $g_m<11$, the map $\Psi_{g_m}$ is not generically quasi-finite and Mukai's program will fail.  We expect that Theorem \ref{mainthm} holds whenever $g_m\geq 11$. So far, the missing values of $(g,m)$ are 
\begin{equation*}
     (2,m) \text{~with~} m\geq 4,\quad (3,3),\quad  (3,4),\quad (4,2),\quad  (4,3),\quad (5,2),\quad  (6,2).
\end{equation*}
A mysterious case is when $g=2$, where our method fails for any $m$. 

More generally, one may consider the restriction map \eqref{eq_restriction map} for $v^2=2n>0$. Most recently, Feyzbakhsh \cite{Feyzbakhsh2022} has  generalised her construction in \cite{Feyzbakhsh2020, Feyzbakhsh2020+} and showed that  for each Mukai vector $v=(r,c,s)$ satisfying \begin{equation}\label{eq:Fey-cond}
     c<r, \quad \gcd(r,c)=\gcd(c,s)=1\quad \text{and}\quad -2\le v^2\leq -2+r,
\end{equation}  the restriction gives an isomorphism $\bM(v)\cong \BN_C(v)$ when $g$ is sufficiently large and the class of $C$ is primitive.  As mentioned in \cite{Feyzbakhsh2022},  the analysis in \cite{Feyzbakhsh2022} also works for the non-primitive case and one can actually show that Feyzbakhsh's construction still gives an isomorphism for   $ C\in |mH| $  if $g$ is sufficiently large (depending on $r$ and $m$).  This gives many examples of  Brill--Noether loci on curves as hyper-K\"ahler varieties of dimension $2g-2r\lfloor \frac{g}{r}\rfloor$. In this paper, we  also improve her result (see Theorem \ref{thm:iso}) and obtain an explicit condition  of $v$  for  $\psi$ being an isomorphism (see Theorem \ref{thm_isomorphism HK}).  As an application,  we show that one can construct hyper-K\"ahler varieties as the Brill--Noether loci of curves in every dimension.  
\begin{theorem}\label{mainthm2}
For any $n>0\in\ZZ$, there exists an integer $N=N(n)$  satisfying that  if $g>N$, there is a primitive Mukai vector  $v\in \Halg(X)$ with $v^2=2n$ such that the restriction map 
$\psi:\bM(v)\to \BN_C(v)$
is an isomorphism for all  irreducible curves $C$ on $X$.
\end{theorem}
The strategy of our proof is  similar to  \cite{Feyzbakhsh2020,Feyzbakhsh2020+,Feyzbakhsh2022}. Roughly speaking, we prove that $\psi$ will be a  well-defined and injective morphism if the Gieseker chambers for objects with  Mukai vector $v$ and $v(-m)\defeq e^{-mH} v$ are large enough, and $\psi$ is surjective if  the Harder--Narasimhan polygon  of $i_\ast F$ for $F\in\BN_C(v)$  achieves its maximum.  The main ingredient is the use of wall-crossing argument to analyse the existence of walls.  There are two crucial improvement in our approach. One is that we find the strongest criterion (Proposition \ref{prop:triangle rule+}) to characterise the stability conditions not lying on the wall of objects with a given Mukai vector. This leads to a more explicit condition for $\psi$ being an isomorphism. The other one is that we develop a method in analysing  the relative position of HN polygons towards  the surjectivity of $\psi$.  This allows us to get a sharper bound of $(g,m)$  without using the computer program. 

\subsection*{Organization of this paper}
In Section \ref{sec:preliminary}, we review the basic knowledge of the Bridgeland stability condition on K3 surfaces and the wall-chamber structure on a section. In Section \ref{sec:des-region},  we introduce the (strictly) destabilising regions $\Omega^{(+)}_v(-)$ of a Mukai vector $v$. They exactly characterise the stability conditions  which are not lying on the wall of objects in $\rD^b(X)$ with Mukai vector  $v$. This will play a key role in the proof of our main theorems.  

In Section \ref{sec:resriction}, we show that the map $\psi:\bM(v)\to \BN_C(v)$ is a well-defined morphism  and $h^0(X,E)=r+s$ for any $E\in \bM(v)$  if  the positive integers $r,c$ and $s$ satisfy 
\begin{equation}\label{eq:int-inj}
    \gcd(r,c)=1,\, r>\frac{v^2+2}{2}\,\text{ and } \,
    g-1\geq \begin{cases} r, & \hbox{if $v^2=0$;}\\ \max \big\{ \frac{r^2}{c}>r ,\frac{r^2}{mr-c}\big\}, & \hbox{if $v^2>0$.}
    \end{cases}
\end{equation}
 The first two assumptions provide that any stable sheaf in $\bM(v)$ is locally-free while the third assumption essentially  ensures that there is no wall between the large volume limit and the Brill--Noether wall.  As a by-product, we obtain a numerical criterion for the injectivity of $\psi$. 

Section \ref{sec:surj} and Section \ref{sec:surj2} are devoted to studying the surjectivity of the restriction map $\psi$.  
They contain the  most technical part of this paper. In Section \ref{sec:surj}, we show that $\psi$  is surjective if the  Harder--Narasimhan polygon of $i_\ast F$ for arbitrary $F\in \BN_C(v)$ is maximal when $g$ is relatively large.
It involves a dedicated analysis of the slope of destabilising factors of $i_*F$  via a geometric vision of the destabilising region.  In Section \ref{sec:surj2},  we analyse the sharpness of HN-polygons for special Mukai vectors with zero square.  The concept of sharpness is used to detect how far the HN-polygon stays away from the convex polygon given by the critical position of the first wall. This makes the construction valid for small genera.

In Section \ref{sec:iso}, we analyse the surjectivity of the tangent map  $d\psi$ of  $\psi$ and show that it is always surjective if $g-1\geq 4r^2$.  In Section \ref{sec:proofofthm}, we prove Theorem \ref{mainthm} and Theorem \ref{mainthm2} by showing the existence of Mukai vectors satisfying all conditions. Here, we make use of the bound of prime character nonresidues.

\subsection*{Notation and conventions} Throughout this paper, we always assume  $(X, H)$ is a primitively polarised K3 surface of genus $g$ of Picard number one.  

For any two points $p,q\in\RR^n$,  let $\rL_{p,q}$ be the line passing through them and let $\rL^+_{p,q}$ be the ray starting from $p$.  We use $\rL_{[p,q]}$, $\rL_{(p,q)}$, $\rL_{(p,q]}$ and $\rL_{[p,q)}$  to denote the closed, open, and  half open line segment respectively. For  any line segment $I$, we set $$\tri_{p}(I)=\bigcup\limits_{q\in I} \rL_{(p,q]} $$ to be a (half open) triangular region.  We denote by $\bP_{p_1\dots p_n}$ the polygon with vertices $p_1,\dots, p_n$.  

\subsection*{Acknowledgements}
The authors want to thank Yifeng Liu for very helpful discussions. This project was initiated from the reading workshop ``Curves on K3 surfaces" in June 2021 supported by NSFC General Program (No. 12171090). 

The authors were supported by NSFC for Innovative Research Groups (No.~12121001) and Shanghai Pilot Program for Basic Research (No. 21TQ00). 

\section{Stability condition on K3 surfaces}\label{sec:preliminary}

Let $\dCat(X)$ be the bounded derived category of coherent sheaves on $X$. 
We let $\nGgp(X)$ be the Grothendieck group of $X$ modulo numerical equivalence. It is onto the (algebraic) Mukai lattice  $\Halg(X) \defeq \rH^0(X,\ZZ) \oplus \NS(X) \oplus \rH^4(X,\ZZ)$ via the map
\[v(E)  = \ch(E)\sqrt{\td(X)}\in \Halg{(X)}.\]
 As $X$ has Picard number one, we may identify   $\Halg(X)$ as $\ZZ^{\oplus 3}$ and  write $v(E)=(r,c,s)$ with $r=\rk(E)$,  $c_1(E)=cH$ and $s=\chi(E)-r$. The Mukai pairing $\left<,\right>$  on $\Halg(X)$  can be viewed as an integral quadratic form on $\ZZ^{\oplus 3}$ given by 
 $$\left <(x,y,z), (x',y',z') \right > =2yy'(g-1)-xz'-zx'.$$
 We may write $v^2=\left<v,v\right>$ for  $v\in \Halg(X)$. 
Consider the projection map
\[\begin{aligned}
     \mathrm{pr} \mcolon \Halg(X)\otimes \RR \setminus \{s=0\}  \rightarrow \RR^2
\end{aligned}
\] 
sending a vector $v=(r,c,s)$ to $(\frac{r}{s},\frac{c}{s} )$. We write $\pi_v=\mathrm{pr}(v)$ and  $\pi_E=\mathrm{pr}(v(E))$ for $E\in \rD^b(X)$ for simplicity.  We let  $O=(0,0,0)$  be the origin of $\Halg(X)\otimes \RR$ and denote by $o=(0,0)$ the origin of $\RR^2$. 
 
 A \textbf{numerical (Bridgeland) stability condition} on $X$ is a pair $\sigma = (\cA_\sigma, Z_\sigma)$ consisting a heart $\cA_{\sigma} \subset \dCat(X)$ of a bounded t-structure and an $\RR$-linear map 
$$Z_\sigma \mcolon \nGgp(X)\otimes \RR \rightarrow \CC $$
satisfying the  conditions
\begin{enumerate}
\item For any $0\neq E\in \cA$,
    \[
        Z_{\sigma}(E)\in \RR_{>0} \exp(i \pi \phi_{\sigma}(E)) \;\text{with} \; 0 < \phi_{\sigma} (E) \leq 1
    \]  where $\phi_{\sigma}(E)$ is the \textbf{phase} of $Z_{\sigma}(E) $ in the complex plane. 
    
 \item The Harder--Narasimhan (HN) Property, cf. \cite[Definition~2.3]{Bridgeland2007}.
\end{enumerate}
 The $\sigma$-slope is defined by \[
        \mu_{\sigma} (E)  = -\frac{\re Z(E)}{\im Z(E)}
    \] and we set the $\sigma$-phase to be \[
        \phi_{\sigma} (E) = \frac{1}{\uppi}[\uppi - \cot^{-1} (\mu_{\sigma} (E))] \in (0,1].
    \]  
     An object $E \in \cA $ is called \textbf{$\sigma$-(semi)stable} if $\mu_{\sigma} (F) < (\leq) \,\mu_{\sigma}(E)$ or equivalently
$   \phi_{\sigma}(F)  < (\leq) \, \phi_{\sigma} (E)$
 whenever $F \subset E$ is a subobject of $E$ in $\cA$. 
We say an object $E \in \dCat(X)$ is {\bf $\sigma$-(semi)stable} if $E[k] \in \cA$ for some $k$, and $E[k]$ is $\sigma$-(semi)stable. 

If $E$ is a sheaf, we write $\mu_H(E)= \frac{c}{r}$ as the $H$-slope of $E$.  There is an associated phase function $\phi_H(E)$, which  can be viewed as the limit of $\phi_\sigma$ by tending $\sigma$ to $o$. We write $\mu_H^{\pm}(E)$ for the $H$-slope of the first/last HN factor of $E$. 

In \cite{Bridgeland2008},  Bridgeland has constructed a continuous family of stability conditions  on $X$ as follows:  for $\alpha,\beta \in \RR$ with $\alpha > 0$,
for any $\beta \in \RR$,  the $\beta$-tilt of $\Coh(X)$ is defined by
 \[
 \begin{split}
     \Coh^\beta (X) \defeq \Big\{ E\in \dCat(X) \,\Big|\,  \mu_H^+ (\coho^{-1}(E)) \leq \beta , \mu_H^- (\coho^0(E))  > \beta,\, \coho^i(E) = 0 \text{ for } i\neq 0,-1  \Big\}
 \end{split}
 \] which is the heart of a t-structure on $\dCat(X)$  with 
\[ Z_{\alpha ,\beta}(E) =  \pairing{(1,\beta ,\frac{H^2}{2}(\beta^2-\alpha^2))}{v(E)} + \sqrt{-1}\pairing{(0,\frac{1}{H^2},\beta)}{v(E)},\]

\begin{theorem}\label{thm:stab}\cite{Bridgeland2008}
  The pair $\sigma_{\alpha,\beta}\defeq (\Coh^\beta(X),Z_{\alpha,\beta})$ 
is a Bridgeland stability condition on $\dCat(X)$ if $\re Z_{\alpha,\beta}(\delta) >0$ for all roots $\delta \in \rR(X)$ with $\rk(\delta)>0$ and $\im Z_{\alpha,\beta}(\delta) =0$.  
\end{theorem} 

The stability condition $\sab$ is uniquely characterised by its kernel \[
    \ker Z_{\alpha,\beta} = \mybrace{
  (r,c, s)\in  \Halg(X)}{c=r\beta, s = \frac{r H^2}{2}(\alpha^2+\beta^2)}.
\]  
According to \cite[Lemma 2.4]{Feyzbakhsh2020}, if we  set  $ k(\alpha,\beta)= \mathrm{pr}(\ker Z_{\alpha, \beta}) \in \RR^2$,  then  $k(\alpha,\beta)$ are parameterised by the space
\begin{equation} \label{eq_structure of stab_condit }
V(X) = \mybrace{(x,y)\in \RR^2}{x>\frac{H^2 y^2}{2}} \setminus \bigcup_{\delta \in \rR(X)} \rL_{(\pi'_\delta, \pi_\delta]}
\end{equation}
where $\rR(X)$ is the collection of roots in $\Halg(X)$ and $\pi'_\delta$ is the intersection point of the parabola $\Big\{x=\frac{H^2 y^2}{2}\Big\}$ and the line $\rL_{o, \pi_\delta} $. Therefore, we may view the stability condition $\sigma_{\alpha,\beta}$ as a point in  $V(X)$. The following are some simple observations that will be frequently used in this paper:

\begin{enumerate}[label=(\Alph*)]
    \item \label{obs-A} If $\sigma\in V(X)$, then the line segment $\rL_{( o, \sigma]}$ is contained in  $V(X)$.
    \item \label{obs-B} If $\gcd(r,c)=1$ and $r>0$, the line $ry=cx$ contains a (unique) projection of root if and only if $r \mid c^2(g-1)+1$ (cf.~ \cite{Yoshioka1999}).  
In particular, the unique projection of root on the $x$-axis is $(1,0)$, which we denote by $o'$.
\end{enumerate}

A simple observation is, for elements in the same heart, we can read their phases from the plane.

\begin{proposition}[Phase reading]\label{prop_phasereading}
Fix $\sigma_{\alpha,\beta} \in V(X)$. For $E\in \Coh^\beta(X)$, let $0<\theta_\sigma\leq \uppi$ be the directed angle from $\overrightarrow{\sigma \pi_E}$ to $\overrightarrow{o\sigma}$ modulo $\uppi$.
Then, $\phi_\sigma$ is a strictly increasing function of $\theta_\sigma$. 
\begin{figure}[ht]
\centering
\begin{tikzpicture}[domain=-3:3,trim right = 4cm]

	\draw (0,0) node [below] {$o$};
	\filldraw [black] (7/6, 22/7) circle [radius=1pt] node [below right=-1 pt] {$\pi_E$};
	\filldraw [black] (-0.75, 1.5) circle [radius=1pt] node [below=3pt] {$\sigma$};

	\draw[->] (-3,0) -- (3,0) node [right]{$y=c/s$};
	\draw[->] (0,0) -- (0,4) node [above]{$x= r/s$};

	\draw plot (\x,{0.45*(\x)^2}) node[right] {$x =\frac{H^2}{2}y^2$};
	\draw [densely dotted] plot (\x,{(6/7) *(\x)+(15/7)}) ;
	\draw [densely dotted]  (-2.5,5)  -- (0,0)  ;
	\coordinate (b) at (-0.75-1, 1.5+2);
	\coordinate (v) at (-0.75, 1.5);
	\coordinate (a) at (-0.75+0.7, 1.5+0.6);
	\pic["$\theta_\sigma$", draw, <-, angle eccentricity=1.5, angle radius=0.6cm]
    {angle=a--v--b};
\end{tikzpicture}
\caption{An example of $\theta_\sigma(E)$.}
\label{fig:Phase reading}
\end{figure}
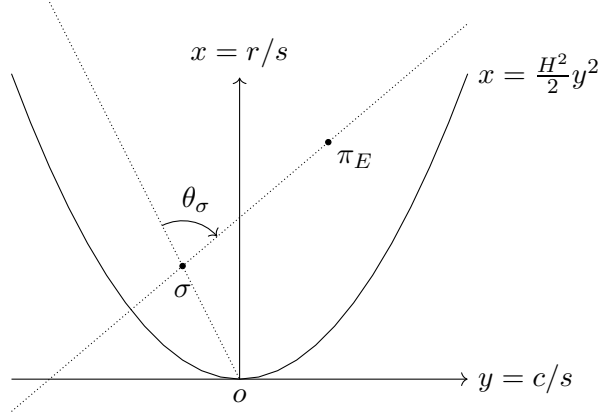
\end{proposition} 

\begin{proof}
Note that $\phi_\sigma(E_1)=\phi_\sigma(E_2)$ if and only if
\begin{equation*}
      v(E_1) + \lambda v(E_2) \in  \ker Z_{\alpha, \beta} 
\end{equation*}
for some $\lambda \in \RR^*$, which is equivalent to $\sigma,\pi_{v(E_1)},\pi_{v(E_2)}$ being colinear, as $\sigma\in V(X)$ is precisely the projection of the kernel of $Z_\sigma$. 
This already proves $\phi_\sigma$ is a strictly monotonic function of $\theta_\sigma$ due to continuity.
It is increasing since $ \phi_{\sigma}(0+)< \phi_{\sigma}(\pi)$. 
The interchange phase $\phi_\sigma=1$ corresponds to the line $\rL_{o, \sigma}$.
\end{proof}

\subsection*{Wall and chamber structure} For any object $E\in \dCat(X)$, there is a wall and chamber structure of $V(X)$ described as follows. 

\begin{proposition}[cf.~{\cite[Proposition 2.6]{Feyzbakhsh2020}}] \label{prop_chamber decomposition}
Given an object $E\in \dCat(X)$, there exists a locally finite set of walls
(line segments) in $V(X)$ with the following properties:
\begin{enumeratea}
    \item The $\sab$-(semi)stability  of $E$ is independent of the choice of the stability condition $\sab$ in any chamber. 
    \item If $\sigma_{\alpha_0, \beta_0}$ is on a wall $\cW_E$, i.e. the point $k(\alpha_0, \beta_0)\in \cW_E$,  $E$ is strictly $\sigma_{\alpha_0, \beta_0}$-semistable.
    \item If $E$ is semistable in one of the adjacent chambers to a wall, then it is unstable in other adjacent chambers.
    \item  Any wall $\cW_E$ is a connected component of $\rL \cap V (X)$, where $\rL$ is a line  passing through the point $\pi_E$ if $\chi(E)\neq \rk(E) $, or with slope $\rk(E)/c_H(E)$ if $\chi(E) = \rk(E)$.
\end{enumeratea}
\end{proposition}
By definition, if $E\in \Coh^\beta(X)$  is $\sigma_{\alpha,\beta}$-semistable, then $\pi_E\neq \sigma_{\alpha,\beta}$.
Combined with Proposition~\ref{prop_chamber decomposition},  one can see that for any line segment $\rL_{[ \sigma_1, \sigma_2 ]}\subseteq V(X)$ containing   $\sigma_{\alpha,\beta }$  with $\sigma_1, \sigma_2$, and $\pi_E$  colinear, we have 
\begin{equation}\label{eq_def}
    \pi_E\notin \rL_{[\sigma_1, \sigma_2 ]}.
\end{equation}
This will be used in later sections. 

\section{The destabilising regions}\label{sec:des-region}
In this section, we characterise the stability conditions which are not lying on the walls of an object  $E\in \rD^b(X)$. As a warm up, we first assume  $\pi_E\in \partial V(X)$ and hence $v(E)^2=0$ or $-2$. Then we have
\begin{proposition}[Triangle rule]\label{prop_triangle rule}
Let $E\in \dCat(X)$ and let $I \subseteq V(X)$ be a line segment.  Assume 
\begin{equation}\label{eq_triangle rule}
    \tri_{\pi_E}(I) \subseteq V(X).
\end{equation}
Then any point in $I$ is not on a wall. In particular, if $I=\rL_{[ \sigma_1,\sigma_2]}$, then $E$ is $\sigma_1$-stable if and only if it is $\sigma_2$-stable. 
\begin{figure}[ht]
\centering
\begin{tikzpicture}[domain=-2.7:2.7,trim right = 4cm]
    
	\draw (0,0) node [below] {$o$};
	\coordinate (E) at (2, 1.8);
	\coordinate (p1) at (-1.5, 3);
	\coordinate (p2) at (-1, 1.5);
	\draw [blue,opacity=0.4] (p1) -- (p2);
	\coordinate (p0) at ($ (p1)!.5!(p2)$);
	\filldraw [black] (E) circle [radius=1pt] node [below right] {$\pi_E$};
	\filldraw [black] (p1) circle [radius=1pt] node [left] {$\sigma_1$};
	\filldraw [black] (p2) circle [radius=1pt] node [below left] {$\sigma_2$};
	\draw[red] (p0) -- (E);
	\draw[->] (-2.7,0) -- (2.7,0) node [right]{$y$};
	\draw[->] (0,0) -- (0,4) node [above]{$x$};

	\draw plot (\x,{0.45*(\x)^2}) node[above] {$x =\frac{H^2}{2}y^2$};
	\filldraw [fill = blue!20!white,draw = blue,opacity=0.4] (p1)  -- (E) -- (p2);
	\filldraw [black] (p0) circle [radius=1pt] node [left] {$\sigma_0$};
\end{tikzpicture}
\caption{}
\label{fig:Triangle Rule}
\end{figure}
\end{proposition}
\begin{proof}
Assume on the contrary, i.e. there is a wall $\cW_E\subseteq \rL\cap V(X)$ where L passes through $\pi_E$ and intersects with $I$. Let $\sigma_0=I \cap \cW_E$. By our assumption, one has  
$$\rL_{(\pi_E,\, \sigma_0 ]} \subseteq \cW_E \subseteq V(X).$$
By Proposition \ref{prop_chamber decomposition} (b),  $E$ is strictly $\sigma$-semistable for any $\sigma\in \rL_{(\pi_E , \sigma_0 ]}$. 
Up to a shift, one may assume that $E\in \Coh^{\beta(\sigma_0)}(X)$.   Since $\sigma_0$ is on a wall, there exists some semistable factor $F\subset E$ in $\Coh^{\beta(\sigma_0)}(X)$ such that $\phi_{\sigma_0}(F)=\phi_{\sigma_0}(E)$ and $\phi_{\sigma}(F)>\phi_{\sigma}(E)$ for $\sigma$ in an adjacent chamber. In particular, $\pi_F \neq  \pi_E$. 
Applying Proposition \ref{prop_chamber decomposition}(b) to $E$, $F$, and $\cok(F \to E)$ respectively, we know that they remain in the heart for any $\sigma\in  \rL_{( \pi_E,\sigma_0]}$.  Hence $F\subset E$ is a proper subobject in the corresponding $\Coh^\beta(X)$. As a consequence, we get
$$0 <  |Z_\sigma(F)| < | Z_\sigma(E)|. $$
Now, if we tend $\sigma$ to $\pi_E$, then $| Z_\sigma(E)| \to 0$ while $| Z_\sigma(F)|\to \epsilon >0$ since $\pi_F \neq \pi_E$. This is a contradiction.
\end{proof}

\subsection*{Destabilising regions}
The Proposition above only works for  $\pi_E\in \partial V(X)$ due to \eqref{eq_def}.
For the case $v(E)^2\geq 0$,  we need to make use of the three-dimensional region defined  as below:   for any $\sigma \in V(X)$ and $v\in \Halg(X)$, 
let $\rL_{( \sigma',\sigma'')} \subseteq \rL_{\sigma, \pi_v} \cap\,  V(X)$ be the connected  component  containing $\sigma$. 
Let $[\sigma]\subseteq \RR^3$ be the preimage of $\sigma$ via the projection $\mathrm{pr}:\RR^3\dashrightarrow\RR^2$. We define the \textbf{destabilising region of $v$ with respect to $\sigma$} as
$$\Omega_v(\sigma) =   (\bP_{Ov^{\!+}_{\sigma} v v^{\!-}_{\sigma}} \setminus \{o,v\} ) \cap \left \{u\in \RR^3\mid u^2\geq -2, (u-v)^2\geq -2 \right \},$$
where $ v^+_\sigma=  [\sigma'] \cap ( [ \sigma'']+v) $ and $v^-_\sigma=[\sigma''] \cap ( [\sigma']+v)$. Up to switching $\sigma'$ and $\sigma''$, we may always let $$v^+_\sigma\in \{(x,y,z)|~x\geq 0, z\geq0 \}\quad\text{and}\quad v^-_\sigma \in \{(x,y,z)|~ x\leq 0, z\leq0\}.$$  There is a natural decomposition 
$$\Omega_v(\sigma) = \Omega_v^+(\sigma)\sqcup \rL_{(O,v)}\sqcup \Omega_v^-(\sigma),$$
where  $\Omega^\pm_v(\sigma)=\Omega_v(\sigma)\cap \bP_{Ov^{\!\pm}_{\sigma} v} \setminus \rL_{(O,v)}$. 
We call $\Omega^+_v(\sigma)$ the \textbf{strictly destabilising region of $v$ with respect to $\sigma$}.   A key result is

\begin{lemma}\label{lemma_triangle rule+}
For $E\in \dCat(X)$ and $\sigma \in V(X) $,  if $\sigma$ is lying on a wall of $E$, then there exists an integer point  in $ \Omega^+_{v(E)} (\sigma) .$ 
\end{lemma}

\begin{figure}[ht]
\label{}
\centering
\begin{tikzpicture}[domain=-6:6,trim right = 6cm,samples = 100]

	\coordinate (O) at (-3,0) ;
	\coordinate (v) at (3,1);
	\coordinate (v+) at (0,7);
	\filldraw [black] (O) circle [radius=1pt] node [below left] {$O$};
    
    \filldraw [black] (v) circle [radius=1pt] node [below right] {$v$};
    
    \filldraw [black] (v+) circle [radius=1pt] node [above right] {$v^+_{\sigma}$};
    
    \draw[] (O) -- (v) --(v+)--(O);
    \path [name path = baseline] (O)--(v);
    
    	%
		%
		\newcommand\tikzhyperbola[9][thick]{%
			\draw [#1, rotate around={#2: (0, 0)}, shift=#3]
			plot [variable = \t,  domain=-#6:#7] ({#4 / cos( \t )}, {#5 * tan( \t )}) node [black][#9] {#8};
		}
		
		\def\angle{90}
		\def\bigaxis{1cm}
		\def\smallaxis{0.8cm}

		\tikzhyperbola[name path = hyperbola_1, line width = 1.2pt, color=blue!80!black]{\angle}{(O)}{\bigaxis}{\smallaxis}{82}{70}{$u^2=-2$}{above};
		
		\tikzhyperbola[name path = hyperbola_2, line width = 1.2pt, color=blue!80!black]{\angle}{(v)}{\bigaxis}{\smallaxis}{70}{80}{$(u-v)^2=-2$}{below right =4.5 and 5};
		
		\pgfmathsetmacro\axisratio{\smallaxis / \bigaxis}
		
		

	 \filldraw[name intersections={of=hyperbola_1 and hyperbola_2, by=w}] [black] (w) circle [radius=1pt] node [ left] {$w$};
	 \path[name path = sline] (v+) -- ($(w)!-1.8!(v+)$);
	 \draw [name intersections={of=sline and baseline, by=q}] [thick,orange] (v+) -- (q) ;
    \filldraw [black] (q) circle [radius=1pt] node [below] {$q$};
    
     \fill[fill = black!40!white,opacity = 0.6] (w) -- ($(w)!-0.7!(v)$) -- ($(w)!-0.7!(O)$)  --cycle;
     
     \coordinate (p1) at ($(O)!.5!(q)$);
     \coordinate (p2) at ($(v)!.5!(q)$);
     \filldraw [black] (p1) circle [radius=1pt] node [below] {$p_1$};
     \filldraw [black] (p2) circle [radius=1pt] node [below] {$p_2$};

    \draw (p1) -- ($(p1)+($(v+)-(q)$)$) ;
    \draw (p2) -- ($(p2)+($(v+)-(q)$)$) ;
    \filldraw [black] ($(p1)+0.4*($(v+)-(q)$)$) circle [radius=1pt];
    \filldraw [black] ($(p2)+0.4*($(v+)-(q)$)$) circle [radius=1pt];
\end{tikzpicture}
\caption{}
\label{fig:Q-integer point}
\end{figure}

\begin{proof}
Set $v=v(E)$ for temporary notation. Firstly, for any $G\subseteq F \subseteq E$ in $\Coh^{\beta(\sigma)}(X)$ satisfying that $E,F,G$ have the same $\sigma$-phase,
we always have that $v(F/G)$ is lying in the parallelogram $ \bP_{Ov^{\!+}_{\sigma} v v^{\!-}_{\sigma}}$.
This is because  there are inclusions $$ \rL_{[0, Z_{\tau}(F/G) ]} \subseteq  \rL_{[ 0, Z_{\tau}(F) ] } \subseteq \rL_{[0, Z_{\tau}(E) ]} $$ for any $\tau \in  \rL_{(\sigma',\sigma'' )}$, which yields that
\begin{equation*}
 v(F/G) \in     \bigcap_{\tau \in \rL_{( \sigma', \sigma'' )}} Z_\tau^{-1}(\rL_{[ 0,Z_\tau(v)]})=\bP_{Ov^+_{\sigma} v v^-_{\sigma}}.
\end{equation*}
In particular, if $0=\widetilde{E}_0 \subset \dots \subset \widetilde{E}_k=E$ is a $\sigma$-JH filtration of $E$ with $E_i = \widetilde{E}_i / \widetilde{E}_{i-1}$ its JH-factors, then any $v(E_i)$ and also   $v-v(E_i)$ is lying  $\bP_{Ov^+_{\sigma} v v^-_{\sigma}}$.

If necessary, reordering these factors $E_i$ such that
the angles between $\rL^+_{O,v(E_i)}$ and $\rL_{O,v_\sigma^+}^+$ increase.
Then  $\sum\limits_{i=1}^{j} v(E_i)$ is an integer point in $\bP_{O v_\sigma^+ v}\setminus \rL_{(O,v)} $ for any $j$. We claim that  either $v(E_1)$ or $\sum\limits_{i=1}^{k-1} v(E_i)$ is lying in $\Omega_v(\sigma)$. This can be proved by using purely Euclidean geometry. 
Suppose this fails, then  we have
\begin{equation}\label{eq:position}
\begin{aligned}
  v(E_1)^2\geq -2 \quad \hbox{and}\quad  (v(E_1)-v)^2  <-2 ,   \\
  (\sum\limits_{i=1}^{k-1} v(E_i) -v)^2\geq -2 \quad \hbox{and}\quad  (\sum\limits_{i=1}^{k-1} v(E_i) )^2<-2,
\end{aligned}
\end{equation}
as $E_i$ is stable. If we restrict the quadratic equation $u^2=-2$ to the plane of $\bP_{ovv_\sigma^+}$, we can obtain a hyperbola, whose center is $O$. The edge  $\rL_{[O,\, v_\sigma^+]}$ can meet the connected component of this hyperbola at most one point. Similarly, $\rL_{[v,\, v_\sigma^+]} $ can intersect with the connected component of  the hyperbola defined by $(u-v)^2=-2 $ at most one point. Note that the edge $\rL_{[O,v]}$ is lying outside the area   
\begin{equation}\label{eq:region}
    \Big\{u^2<-2, (u-v)^2<-2\Big\}.
\end{equation}
see the shadow part in Figure \ref{fig:Q-integer point}. By \eqref{eq:position}, $v_\sigma^+$ has to lie  in the region \eqref{eq:region}. 
Moreover,  as one can see from  the picture, there is a point  $w\in \bP_{Ov^+_\sigma v}^\circ$ lying on the intersection of two hyperboloids
\begin{equation}
\Big\{u\in \RR^3|~ u^2= -2\Big\} \cap \Big\{u\in\RR^3|~(u-v)^2=-2\Big\}. 
\end{equation}
and the line $\rL_{w,v^+_\sigma}$  will intersect  the edge $\rL_{[O,v]}$  at a point, denoted by  $q$. 
Thus we get that the point $v(E_1)$ is contained in the triangle $\bP_{qv v^+_\sigma}$,  while the point $\sum_{i=1}^{k-1} v(E_i)$ is contained  in the triangle $ \bP_{Oq v^+_\sigma  }$.  If we define $\rL_u$ to be the line passing through  $u\in\RR^3$ and parallel to the line $\rL_{w, v^+_{\sigma} }$, the discussion above exactly means 
\begin{equation}\label{eq:inclusions} 
\rL_{[O,\, p_1]} \subset \rL_{[O,\,q)}\subset \rL_{[O,\,p_2]}
\end{equation}
where $p_1=\rL_{\sum_{i=1}^{k-1} v(E_i)}\cap \rL_{[O,\,v]} $  and  $p_2=\rL_{v(E_1)}\cap \rL_{[O,\,v]}$, see Figure \ref{fig:Q-integer point}.

Next,  one can regard $\rL_{w,v^+_\sigma}$ as a stability condition  in a natural way (by taking the projection of the one dimensional vector space determined by $\rL_{w,v^+_\sigma}$). Denote this stability condition by $\tau$. Then  for any two points $p,p'$  with  $\rL_{p,\,p'}$ parallel to $\rL_{w,v^+_\sigma}$, there is $Z_\tau (p) = Z_\tau (p')$. Using the inclusions \eqref{eq:inclusions}, we obtain the inequality
 \begin{equation}
 \frac{\big\lvert Z_{\tau}(v(E_1)) \big\rvert}{\big\lvert Z_{\tau}(\sum_{i=1}^{k-1} v(E_i))\big\rvert }=\frac{\big\lvert Z_{\tau}(p_2)\big\rvert }{ \big\lvert Z_{\tau}(p_1) \big\rvert }    = \frac{\|\rL_{[O,\,p_2]}\|}{\|\rL_{[O,\,p_1]}\|}>1.
\end{equation}
However, this contradicts to  the relation $ \sum_{i=1}^{k-1}\big\lvert Z_{\tau}(v(E_i)) \big\rvert = \big|Z_{\tau} \left(\sum_{i=1}^{k-1} v(E_i) \right)\big|$ which finishes the proof. 
\end{proof}

\begin{remark}\label{rem_triangle rule+}
If  $v(\widetilde E_i)$  already lies in $\bP_{O v v_\sigma^+} \setminus  \rL_{(O, v )}$ for all $i$, 
then  our  argument actually shows that we can always take a destabilising sequence $$F\hookrightarrow E \twoheadrightarrow Q$$ such that $v(F) \in \Omega^+_{v(E)}(\sigma)$. 
This will happen, for instance,  if $E=i_* F$  for some slope stable vector bundle  $F$  on $C$. 
\end{remark}

Then we can obtain a generalisation of Proposition \ref{prop_triangle rule}. 

\begin{proposition}\label{prop:triangle rule+}
Given a region $\cI\subseteq V(X)$, we define $$ \Omega_{v}(\cI)=\bigcup_{\sigma \in \cI}\Omega_v(\sigma)\quad \hbox{and}\quad \Omega^+_v (\cI) =\bigcup_{\sigma \in \cI}\Omega^+_v(\sigma).$$ Then any $\sigma\in \cI$ is not lying on a wall of any $E$ with $v(E) = v$ if and only if 
\begin{equation}\label{eq_no wall}
    \Omega^+_{v}(\cI)\cap \Halg(X) =\emptyset.
\end{equation}   
Similarly, any $E$ with $v(E) = v$ cannot be strictly $\sigma$-semistable for any $\sigma \in \cI$ if and only if 
\begin{equation}\label{eq_no sss}
    \Omega_{v}(\cI)\cap \Halg(X) =\emptyset.
\end{equation}   
\end{proposition}
\begin{proof}
The `if' part follows directly from Lemma~\ref{lemma_triangle rule+}. For the `only if' part, suppose there exists some stability condition $\sigma$ and an integer point $w \in \Omega_v(\sigma)$. Then, we can find $\sigma$-stable objects $F_1$ and $F_2$ such that $v(F_1)=w$ and $v(F_2)=v-w$, and $\sigma$ will be lying on a wall of $E\defeq F_1 \oplus F_2$ from the  construction. For the strictly semistable case, one just notes that the Mukai vectors of all the factors are lying on $\rL_{(O,v)}$.
\end{proof}

According to Proposition \ref{prop:triangle rule+}, we will say a Mukai vector $v\in \Halg(X)$ \textbf{admits no wall in $\cI$} if \eqref{eq_no wall} holds, and \textbf{admits no strictly semistable condition} if \eqref{eq_no sss} holds. 

Note that from the definition, one automatically has $\Omega_v(\sigma)=\Omega_v(\rL_{(\sigma',\sigma'')})$. This motivates us to find a subregion of $V(X)$ with regular boundary. A candidate is
\begin{equation}\label{eq_Gamma}
    \Gamma = \mybrace{(y,x)\in \RR^2}{x> gy^2, x < \sqrt{2/H^2} \text{ when } y=0 } \subseteq V(X)
\end{equation}
which is used in \cite{Feyzbakhsh2022}. As a consequence, if $v$ admits no wall in $I\subseteq \rL_{(o,o')}$, then it admits no wall in $\tri_{\pi_v}(I) \cap \Gamma$ as well.

\begin{remark}\label{rmk:square zero}
Comparing Proposition~\ref{prop_triangle rule} with Proposition~\ref{prop:triangle rule+}, one can conclude that for $v^2=0$ and $\cI$ being a line segment, the condition \eqref{eq_triangle rule} implies \eqref{eq_no wall}. Actually, if $\Omega^+_v(\sigma)$ contains any integer point $\delta$, then $\pi_\delta$ is a root lying in $\rL_{(\pi_v,\sigma]}$. This suggests that the condition \eqref{eq_triangle rule} can be replaced by $\tri_{\pi_v}(I) \subseteq V'(X)$, where
\begin{equation*}
V'(X) = \mybrace{(x,y)\in \RR^2}{x>\frac{H^2 y^2}{2}} \setminus \bigcup_{\delta \in \rR(X)} \{\pi_\delta\}.
\end{equation*}
\end{remark}




\section{The restriction map to Brill--Noether locus}\label{sec:resriction}

Given a positive primitive vector $v= (r,c,s)\in \Halg(X)$, let  $\bM(v)$ be  the moduli space  of $H$-Gieseker semistable sheaves on the surface $X$ with Mukai vector $v$. In this section,  we always assume  $r,c,s>0$  and 
 \begin{center}
     $\gcd(r,c)=1$ and $r>\frac{v^2}{2}+1$.
 \end{center}  
 Then $\bM(v)$ is a smooth  variety consisting of $\mu_H$-stable locally free sheaves (cf.~\cite[Remark 3.2]{Yo01}). The main result is

\begin{theorem}
\label{thm_restriction map}
For any irreducible curve $C\in |mH|$,  the restriction map is an injective morphism 
 \begin{equation*}
 \begin{aligned}
     \psi:  \bM(v) &\to \BN_C(v) ,\\
     E &\mapsto E|_C
 \end{aligned}
 \end{equation*}
with stable image (i.e.  $E|_C$ is stable) if the following conditions hold
  \begin{itemize}
     \item  $(mr-c)s>rc$;
     \item $v$ admits no wall in $\rL_{(o, \sigma_v]}$, where $\sigma_v=(\frac{rc}{(mr-c)s},0)$;
     \item $v(-m)=\big(r, c-mr, s+ (g-1)m(mr-2c) \big)$ admits no wall in $\tri_{\pi_{v(-m)}}(\rL_{(o, o']}) \cap \Gamma$, where $\Gamma$ is defined in \eqref{eq_Gamma}.
 \end{itemize}   
\end{theorem}

\begin{figure}[ht]
\centering
\begin{tikzpicture}[domain=-3.5:3.5,trim right = 3cm,samples = 100]
    \coordinate (o) at (0,0);
	\draw (o) node [below] {$o$};
	\draw [->] (-3.5,0) -- (3.5,0) node [right]{$y$};
	\draw[->] (0,3) -- (0,5) node [above]{$x$};
	\path [name path = y-axis] (0,-1) -- (0,5);
       	\draw plot (\x,{0.4*(\x)^2}) node[above] {$x =\frac{H^2}{2}y^2$};
	
	\coordinate (E) at (3.5,4);
	\coordinate (F) at (-3,0);
    \filldraw [black] (E) circle [radius=1pt] node [below right] {$\pi_v$};
	\filldraw [black] (F) circle [radius=1pt] node [below] {$\pi_{i_\ast(E|_C)}$};

	\path [name path = link] (F) -- (E);
	
	\coordinate (Em) at ($ (F)!.2!(E)$);
	\draw [densely dotted] 
	        (o) -- (Em);
	\filldraw [black] (Em) circle [radius=1pt] node [above left] {$\pi_{v(-m)}$};

	\draw [name intersections={of=link and y-axis, by=sigma}] 
	        (o) -- (sigma);
	\coordinate (o') at (0,3);
    \draw (o') -- (sigma);
	\draw [densely dotted] (F) -- (E);
	\node[inner sep=0pt]  (sigma_1) at ($ (E)!.95!(0,0) $) {};
    \filldraw [black] (sigma_1) circle [radius=1pt] node [right] {$\sigma_1$};
    \draw [blue,opacity=0.4] (sigma) -- (sigma_1);
    \filldraw[fill = blue!20!white,opacity=0.2] (sigma_1) -- (sigma) -- (0,0) --cycle;

    \filldraw [black] (sigma) circle [radius=1pt] node [above right] {$\sigma_v$};
    \filldraw [black] (o') circle [radius=1pt] node [above right] {$o'$};
    \draw [densely dotted] (Em) --  (o');
    
    \node[inner sep=0pt]  (sigma_3) at ($ (Em)!.9!(0,3) $) {};
    \filldraw [black] (sigma_3) circle [radius=1pt] node [left] {$\sigma_3$};
    
    \node[inner sep=0pt]  (sigma_2) at ($ (Em)!.88!(o) $) {};
    \filldraw [black] (sigma_2) circle [radius=1pt] node [above left] {$\sigma_2$};
    \draw [blue,opacity=0.4] (sigma_2) -- (sigma_3);
    \draw [densely dotted] (o) -- (E);
    \filldraw [fill = blue!20!white,opacity=0.2] (sigma_2) -- (o) -- (o') -- (sigma_3) -- (sigma_2);
    
\end{tikzpicture}
\caption{}
\label{fig:BN Wall}
\end{figure}

\begin{proof} 
It suffices to prove that for any $E\in \bM(v)$, the restriction $E|_C$ is slope  stable with  $h^0(C, E|_C)\geq r+s$ and $E|_C$ uniquely determines $E$.  

Firstly, we show that  $E|_C$ is slope semistable.  
By \cite[Lemma 2.13 (b)]{Feyzbakhsh2020}, it suffices to show that $i_*(E|_C)$ is $\sigma_v$-semistable. 
Consider the  exact sequence 
\begin{equation}\label{eq_ses}
    0 \to E(-C) \to E \to i_\ast (E|_C) \to 0,
\end{equation} 
we have  $\pi_{i_\ast(E|_C)}= (\frac{r}{(g-1)(2c-mr)},0)$
is lying on $\rL_{\pi_{E}, \pi_{E(-C)}}$. 
Since  $E$ is  slope stable,  according to \cite[Lemma 2.15]{Feyzbakhsh2020}, $E$ is $\sigma$-stable for any $\sigma\in \rL_{(o, \pi_v)} \cap V(X) $.
Choose $\sigma_1 \in \rL_{ ( o, \pi_v)}$ sufficiently close to $o$. 
We have $$\bP_{o\sigma_v\sigma_1}\setminus \{o\} \subseteq \Gamma \subseteq  V(X)$$
as in Figure \ref{fig:BN Wall}. Note that for any line $\rL$ passing through $\pi_E$,  the intersection $\bP_{o\sigma_v\sigma_1}\setminus \{o\} \cap \rL$ is connected. 
By our assumption, $v$ admits no wall in $\rL_{(o,\sigma_v]}$. This implies it also admits no wall in  $\bP_{o\sigma_v\sigma_1}\setminus \{o\}$.  Hence $E$ is $\sigma_v$-stable as $E$ is $\sigma_1$-stable.
Similarly, we have $E(-C)$ is also $\sigma_v$-stable.  
As in the proof of  Proposition~\ref{prop_phasereading}, $E$ and $E(-C)$ are of the same $\sigma_v$-phase since $\sigma_v \in \rL_{\pi_{E},\pi_{E(-C)}} $. 
Hence the restriction $i_*(E|_C)$ is $\sigma_v$-semistable with $E$ and $E(-C)[1]$ as its JH-factors.

Secondly,  we show that  $E|_C$ is slope stable. 
By using \cite[Lemma 2.13 (b)]{Feyzbakhsh2020},  we are reduced to prove $i_*(E|_C)$ is $\sigma$-stable for some $\sigma \in \rL_{ (o,\sigma_v)} $. Moreover,  due to \cite[Lemma 2.13 (a)]{Feyzbakhsh2020},
$i_*(E|_C)$ is semistable for any stability condition lying in a  line segment  $ \rL_{( o,a)}\subseteq \rL_{( o,\sigma_v)}  $.
Suppose that $i_*(E|_C)$ is strictly semistable for all stability conditions in $\rL_{(o,a)} $.  
Then for any $\sigma_0 \in \rL_{(o, a)}$ and any destabilising sequence $$F_1 \hookrightarrow i_*(E|_C) \twoheadrightarrow F_2 \in \Coh^{\beta=0}(X)$$  such that $F_1,F_2$ are $\sigma_{0}$-semistable with the same $\sigma_0$-phase as $i_*(E|_C)$, we have   $\pi_{F_1}=\pi_{i_*(E|_C)}$.
This gives $\phi_{\sigma_v}(F_{1})=\phi_{\sigma_v}(i_\ast (E|_C))$, which implies that  $F_1$ is $\sigma_v$-semistable.
However,  this contradicts to the uniqueness of JH-factors of $i_\ast(E|_C)$ with respect to $\sigma_v$. 
Thus $i_*(E|_C)$ is $\sigma$-stable for some $\sigma \in \rL_{(o,a)}$. 

Next, to show $h^0(C, E|_C)\geq r+s$, let us consider the long exact sequence of cohomology induced by (\ref{eq_ses})\[
 0\rightarrow \rH^0 (X,E(-C)) \rightarrow \rH^0 (X,E) \rightarrow \rH^0 (C, E|_C)) \rightarrow \rH^1 (X,E(-C)) \rightarrow \dots   \]
As   $E(-C)$ is $\mu_H$-stable and $ \mu_H (E(-C))<0$, we have $$\rH^0 (X,E(-C))= \Hom_X (\cO_X, E(-C)) =0. $$ 

Then we  choose  $\sigma_2 \in \rL_{(\pi_{v(-m)} , o)}$ sufficiently close to $o$ and $\sigma_3 \in \rL_{(\pi_{v(-m)} , o')}$ sufficiently close to $o'$ so that $\bP_{o\sigma_2\sigma_3o'} \setminus \bigbrace{o,o'} \subseteq \Gamma$, see Figure \ref{fig:BN Wall}.
As shown above,  $E(-C)$ is $\sigma$-stable for any $\sigma\in \bP_{o\sigma_2\sigma_3 o'} \setminus \bigbrace{o,o'}$. 
In particular,  $E(-C)$ is $\sigma_3$-stable. 
According to \cite[Lemma 2.15]{Feyzbakhsh2020}, we have $$\bP_{o'\sigma_3\sigma_v} \setminus \{o'\} \subseteq V(X)$$ and $\cO_X$ is also $\sigma_v$-stable.  
Note that $\pi_{\cO_X} = o'$.
By Proposition \ref{prop_triangle rule} and Proposition \ref{prop_phasereading}, we know that $\cO_X$ is $\sigma_3$-stable and  $\phi_{\sigma_3}(E(-C))=\phi_{\sigma_3}(\cO_X)$.
Then we have 
$$\rH^1 (X,E(-C))= \Hom_\cA (\cO_X, E(-C)[1]) = 0 $$ where $\cA = \Coh^{\beta(\sigma_3)}(X)$. Therefore, we get an isomorphism $\rH^0 (X,E) \xlongrightarrow{\simeq} \rH^0 (C, E|_C)$. 
By Serre duality and the stability of $E$, we have  $\rH^2(X,E) \cong \Hom_X(E,\omega_X) \cong \Hom_X(E,\cO_X) =0 $. 
It follows that 
\begin{equation} \label{eq_global section X}
h^0(C,E|_C)  = h^0(X,E) \geq \chi(E) =r+s  .
\end{equation}  
This proves our claim. 

In the end,  the uniqueness of $E$ follows from the fact that the JH factors of $i_*(E|_C)$ are unique  with respect to $\sigma_v$. 
\end{proof}

\subsection*{A numerical criterion} 
As in \cite{Feyzbakhsh2020}, we would like to find a purely numerical condition for Theorem \ref{thm_restriction map} to hold. An elementary result is  
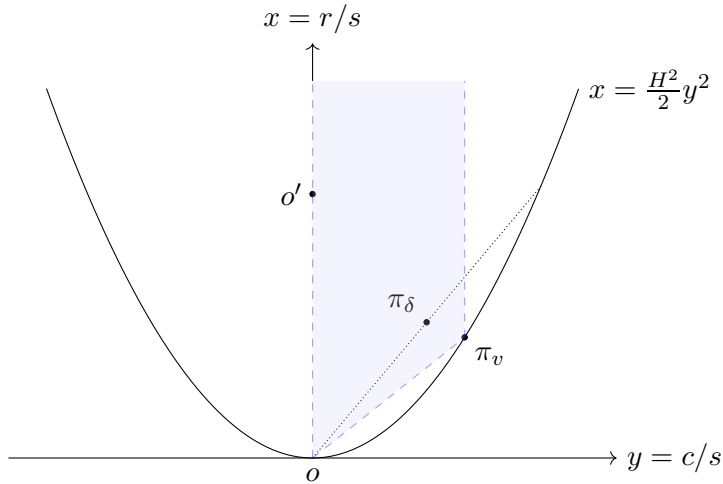
\begin{figure}[ht]
\centering
\begin{tikzpicture}[domain=-3.5:3.5,trim right = 4cm,samples = 100]
	\draw (0,0) node [below] {$o$};
	\filldraw [black] (0,3.5)  circle [radius=1pt] node [left] {$o'$};
	\draw [->] (-4,0) -- (4,0) node [right]{$y=c/s$};
	\draw[dashed,blue,opacity=0.4] (0,0) -- (0,5);
	\draw[->] (0,5) -- (0,5.5) node [above]{$x=r/s$};
	\path [name path = y-axis] (0,-1) -- (0,5);
	\path [name path = parabola] plot (\x,{0.4*(\x)^2}) ;
	\draw plot (\x,{0.4*(\x)^2}) node[right] {$x =\frac{H^2}{2}y^2$};
    \filldraw [black] (1.5, 1.8) circle [radius=1pt] node [above left] {$\pi_\delta$};
    \draw [densely dotted] (0,0) -- (3,3.6);
    \filldraw [black] (2, 1.6) circle [radius=1pt] node [below right] {$\pi_v$};
    \draw[dashed,blue,opacity=0.4] (0,0) -- (2,1.6);
    \draw[dashed,blue,opacity=0.4] (2.,1.6) -- (2,5);
    \fill [fill = blue!20!white,opacity=0.2] (0,0) -- (2,1.6) -- (2,5) -- (0,5) -- cycle;
\end{tikzpicture}
\caption{$\pi_\delta$  in the interior of  $\bP^\circ_{o\pi_v\infty}$ (colored area)}
\label{fig:good region}
\end{figure}

\begin{lemma}\label{lem_no int}
Let $\bP_{o\pi_v \infty}$ be the trapezoidal region bounded by $\rL_{[o,\pi_v]}$, the (positive) half $x$-axis $\rL_{[o,\infty)}$ and the vertical ray $\rL_{[\pi_v,\infty)}$ in Figure \ref{fig:good region}.
Then $v$ admits no wall in $\bP_{o\pi_v \infty} \cap \Gamma$ if one of the following conditions holds
\begin{enumerate}
    \item $v^2=0 \,$ and $\, r/\gcd(r,c) \leq g-1$.
    \item  $ s=\lfloor \frac{(g-1)c^2+1}{r} \rfloor$  and $\,  g-1 \geq \max \{ \frac{r^2}{c}, r+1\}$
\end{enumerate}
\end{lemma}
\begin{proof}
(i).  By Proposition \ref{prop_triangle rule}, it will be sufficient to show that 
$$\bP^\circ_{o \pi_v \infty} \subseteq V(X).$$
Due to the explicit description of $V(X)$ in \eqref{eq_structure of stab_condit },  this is equivalent to showing that there is no projection of root lying in $\bP^\circ_{o\pi_v\infty}$.  Suppose there exists a root $\delta=(r',c',s')\in \rR(X)$ with $\pi_\delta\in \bP^\circ_{o\pi_v\infty}$. Then we have 
\begin{equation}\label{eq:root in region}
\frac{c'}{r' }<\frac{c}{r}\quad\text{and}\quad\frac{c'}{s'}<\frac{c}{s},
\end{equation}
see Figure \ref{fig:good region}. Note that $2rs=c^2(2g-2)$ and $2r's'=(c')^2(2g-2) +2$, one can plug into  \eqref{eq:root in region} to get
\begin{equation} \label{eq:root in region+}
    \frac{r}{\gcd(r,c)} c' < \frac{c}{\gcd(r,c)} r'< \frac{r}{\gcd(r,c)} c' + \frac{r}{\gcd(r,c)(g-1)c'}\leq \frac{r}{\gcd(r,c)} c'+1.
\end{equation}
which is not possible. 

(ii). According to Proposition \ref{prop:triangle rule+}, we just need to show that $ \Omega^+_{v}(\sigma)\cap \Halg(X) =\emptyset$ for any $\sigma \in\bP_{o\pi_v \infty} \cap \Gamma $ .
Suppose there is an integer point $(x,y,z)\in \Omega^+_v( \sigma_0) $ for some $\sigma_0\in \bP_{o\pi_v \infty} $. 
By the construction of $\Omega^+_{v}(\bP_{o\pi_v \infty} \cap \Gamma)$,  we have $0<y\leq c$ and the point $(x,y,z)$ is lying in the interior of the triangle $\bP_{u_1 u_2 u_3}$ with vertices
\begin{center}
    $u_1=(\frac{ry}{c}, y,\frac{gcy}{r})$, $u_2=(\frac{ry}{c}, y,\frac{sy}{c})$  and $u_3=(\frac{gcy}{s}, y,\frac{sy}{c})$. 
\end{center}
  As  $ y^2(g-1)+1\geq xz$ and $z\geq \frac{sy}{c}$,
one has  $$
 \frac{ry}{c}<x< \frac{y^2(g-1)+1}{sy/c}.$$ 
Note that $c^2(g-1)-\frac{v^2}{2}=rs$, the condition $s=\lfloor \frac{(g-1)c^2+1}{r} \rfloor$ is equivalent to  $r>\frac{v^2}{2}+1$.
Then we have 
\begin{equation}
\begin{aligned}
 0< \frac{y^2(g-1)+1}{sy/c}-  \frac{ry}{c} &<  \max\left\{\frac{gc}{s}-\frac{r}{c},  \frac{c^2(g-1)+1}{s}-r\right\}
    \\    & = \max\left\{\frac{r(c^2+\frac{v^2}{2})}{c(c^2(g-1)-\frac{v^2}{2})},  \frac{r(\frac{v^2}{2}+1)}{c^2(g-1)-\frac{v^2}{2}}\right\}\\
    &\leq \max\left\{\frac{r(c^2+r-2)}{c(c^2(g-1)-r+2)},  \frac{r^2-r}{c^2(g-1)-r+2}\right\} \\& \leq  \frac{1}{c}.
\end{aligned}
\end{equation}
Here, the last inequality follows from our assumption $g-1\geq\max\{\frac{r^2}{c},r+1\}$. 
 This means $0<x-\frac{ry}{c}<\frac{1}{c}$
 which contradicts to the fact $x$ is an integer. 
\end{proof}

We summarise our numerical criterion as follows.
\begin{corollary}\label{cor:inj num}
The restriction map $\psi:  \bM(v)\to \BN_C(v)$ is an injective morphism  with stable image if
\begin{equation}\label{eq:inj-cond}
\begin{aligned}
    r>\max\left\{\frac{v^2}{2}+1,\frac{c}{m}\right\}, \quad c>0, \quad s>\frac{rc}{mr-c}, \quad \gcd(r,c)=1\\
\end{aligned}
\end{equation}
and \begin{equation}\label{eq:inj-cond2}
    g-1\geq \begin{cases} r, & \hbox{if $v^2=0$}\\ \max \{ \frac{r^2}{c} ,\frac{r^2}{mr-c},r+1\}, & \hbox{if $v^2>0$.}
    \end{cases}
\end{equation}

\end{corollary}
\begin{proof}
The condition $mr>c>0$ ensures $\pi_v$ lies in the first quadrant while $\pi_{v(-m)}$ lies in the second quadrant, and the condition $s(mr-c)>rc$ ensures $\sigma_v$ is below $o'$. The assertion then follows from  the  direct computation that $\rL_{(o,\sigma_v]} \subseteq \tri_{\pi_w}(\rL_{[o, o']}) \cap \Gamma \subseteq \bP_{o\pi_w \infty} \cap \Gamma $ for $w = v$ or $v(-m)$.
\end{proof}
\begin{remark}
Under the assumption $r>c$, the conditions in Corollary \ref{cor:inj num} can be easily reduced to \eqref{eq:int-inj}. 
\end{remark}

\section{Surjectivity of the restriction map}\label{sec:surj}

Throughout this section, we let $v=(r,c,s)\in\Halg(X)$ be a positive vector satisfying  \eqref{eq:inj-cond} and \eqref{eq:inj-cond2}.  Due to Corollary \ref{cor:inj num},  the restriction map $$\psi:\bM(v)\to \BN_C(v)$$ is an injective morphism with stable image.  Following the ideas in \cite{Feyzbakhsh2020,Feyzbakhsh2020+}, we give sufficient conditions  such that  $\psi$  is surjective.

\subsection*{The first wall}As in \cite{Feyzbakhsh2020}, we first describe the wall that bounds the Gieseker chamber of $i_\ast F$ for $F\in \BN_C(v)$.   The following result is an extension of \cite[Proposition 4.2]{Feyzbakhsh2020}. 
\begin{theorem}\label{thm_1st wall}
For any $F\in \BN_C(v)$,  the wall that bounds the Gieseker chamber of $i_*F$ is not below the line  $\rL_{\pi_v, \pi_{v(-m)}}$, and they coincide if and only if $F=E|_C$ for some $E\in \bM(v)$.
\end{theorem}
\begin{proof}
We will prove that for any $v$ under the condition \eqref{eq:inj-cond}, if both $v$ and $v(-m)$ admit no wall in $(o, \sigma_v ]$, then so does $$v|_C\defeq v-v(-m).$$ Let $\cW_{i_*F}$ be the first wall and let $\sigma_{\alpha',0}\in \cW_{i_*F} $ be a stability condition. 
Suppose $\cW_{i_\ast F}$ is below or on  the line $\rL_{\pi_v, \pi_{v(-m)}}$.
Then there exists a destabilising sequence 
\begin{equation}\label{eq:destab}
   F_1 \hookrightarrow i_*F \twoheadrightarrow F_2 
\end{equation}
in $\Coh^{\beta=0}(X)$  such that $F_1,F_2$ are $\sigma_{\alpha',0}$-semistable, and \begin{equation}\label{eq:wall}
    \phi_{\alpha,0}(F_1)> \phi_{\alpha,0}(i_* F)\quad \text{for} \quad  \alpha<\alpha'.
\end{equation} 
Taking the cohomology of \eqref{eq:destab} gives a long exact sequence of sheaves
\begin{equation}\label{eq_des coho}
     0 \to \coho ^{-1}(F_2) \to F_1 \xrightarrow{d_0} i_* F \xrightarrow{d_1} \coho ^0(F_2) \to 0.
\end{equation}
Set $v(F_1)=(r',c',s')$, then we have $r'>0$ by \eqref{eq:wall}. Let $T$ be the maximal torsion subsheaf of $F_1$ and we can write $v(T)=(0, \hat c, \hat s)$ for some $\hat c, \hat s\in \ZZ$. Consider the inclusions $T \hookrightarrow F_1 \hookrightarrow i_*F$  and take the cohomology, one can get
\begin{equation*}
    0\to \coho^{-1}(\cok) \to T \to i_* F \to \coho^0(\cok)\to 0.
\end{equation*}
Since $\coho^{-1}(\cok)$ is torsion-free, it must be zero. It follows that $T$ is a subsheaf of $i_*F$ and $\rk (i^* T)= \frac{\hat c}{m}$.  If we let $v(\rH^0(F_2))=(0,c'',s'')$, by restricting (\ref{eq_des coho}) to the curve $C$, one can get
\begin{equation*}
\begin{split}
       r'+\frac{\hat c}{m} &=\rk(F_1 / T) + \rk(i^*T )  \geq \rk(i^*F_1)\\
       &\geq \rk(i^*\ker d_1) \geq \rk(i^*F) - \rk (i^*\coho^0(F_2)) = r -\frac{c''}{m}.
\end{split}
\end{equation*} 
In other words,
\begin{equation}\label{eq:<=}
    \mu(F_1/T)-\mu(\rH^{-1}(F_2))=\frac{c'-\hat{c}}{r'} - \frac{c'+c''-m r}{r'} \leq m.
\end{equation}
Using Lemma \ref{lem_factors} below, we can take the destabilising sequence \eqref{eq:destab} satisfying 
\begin{equation}\label{eq:diff}
\mu^-_H(F_1/T)  \geq \frac{c}{r} \,\text{ and } \, \mu^+_H(H^{-1}(F_2)) \leq \frac{c-mr}{r}.
\end{equation}
This gives
\begin{equation}\label{eq:>=}
    \mu^-_H(F_1/T)-\mu^+_H(H^{-1}(F_2)) \geq m.
\end{equation}
Combining \eqref{eq:<=} and \eqref{eq:>=}, we get $mr-c''-\hat{c}= mr'$, thus
\[
\mu_H(F_1/T) = 
\frac{c'-\hat{c}}{r'} = \frac{c'-\hat{c}}{r - \frac{c''+\hat{c}}{m} } = \frac{c}{r} 
\]
and both $F_1/T$ and $ \coho^{-1}(F_2)$ are $\mu_H$-semistable.
Since $\gcd(r,c)=1$ and $i_*F$ does not contain any skyscraper sheaf, we have $\hat{c}= c'' =0$ and  $\hat{s}=0$.
This shows $T =0$ and hence $v(F_1)=(r, 1, s')$.
Note that by our assumption, we have $\pi_{v(F_1)}\in \rL_{(o,\pi_v)}$, which means $s < s'$.
However, this gives  $v(F_1)^2=2r(s-s')< -2$ which  contradicts to the fact that $F_1$ is $\mu_H$-stable.

Assume that  $\cW_{i_\ast \cF}\subseteq \rL_{\pi_v,\pi_{v(-m)}}$. 
Then we have $$\mu^-_H(F_1/T)-\mu^+_H(H^{-1}(F_2)) = m$$ and    $F_1$ is a stable sheaf. Note that the map $d_0:F_1 \rightarrow i_* F$  factors through $d_0'\colon i_* (F_1|_C) \to i_*F$ and $\mu_H(i_* (F_1|_C))=\mu_H(i_\ast F)$. 
Applying Theorem~\ref{thm_restriction map} to $F_1$, we know that  $i_*(F_1|_C)$ is stable as well.
It follows that $d_0'$ is an isomorphism.
\end{proof}

\begin{lemma}\label{lem_factors}
With  notations and assumptions as above, 
one can find a destabilising sequence \eqref{eq:destab} such that $F_i$ satisfies
\begin{equation}\label{eq:2-inequa}
    \mu_H^-(F_1/T)  \geq \frac{c}{r} \quad \text{ and } \quad \mu_H^+(\coho^{-1}(F_2)) \leq \frac{c-mr}{r}.
\end{equation}
\end{lemma}

\begin{proof}
Denote $\sigma_1 = \cW_{i_*F} \cap \rL_{(o,\sigma_v]}$. 
By Remark~\ref{rem_triangle rule+}, we can take the destabilising sequence 
\begin{equation*}
   F_1 \hookrightarrow i_*F \twoheadrightarrow F_2 
\end{equation*}
satisfying  $v(F_i) \in \Omega^+_{v|_C}(\sigma_1) \subseteq \Omega^+_{v|_C} (\rL_{( o, \sigma_v]})$. We divide the proof into three steps. 
 
 \subsubsection*{Step 1} We show that for any point $u=(x_0,y_0,z_0)$ with $u^2 \geq -2$ and $x_0>0$, $u$ is lying in $\Omega_{v}^+(\rL_{(o,\sigma_v]})$ if $x_0\leq r$ or $z_0\leq s$, and $\pi_u\in \bP^\circ_{o\sigma_v\pi_v}$. By its definition,  we know that $u\in \Omega_v^+(\rL_{(o,\sigma_v]})$  if 
 \begin{equation}
    u\in \bP^\circ_{O v v_{\sigma}^+} \quad \text{and}\quad    (u-v)^2 \geq -2 .
 \end{equation}
for some $\sigma \in \rL_{(o,\sigma_v]}$.

As $\pi_u=(\frac{x_0}{z_0}, \frac{y_0}{z_0})$ is lying in the interior of the triangle $\bP_{o\sigma_v\pi_v}$, we have 
\begin{equation}\label{eq:slope-in}
    \frac{y_0}{z_0}< \frac{c}{s}\quad \text{and} \quad \frac{x_0/z_0}{y_0/z_0}=\frac{x_0}{y_0}>\frac{r}{c}.
\end{equation}
The line $\rL_{\pi_u, \pi_v}$ will meet the open edge $\rL_{(o,\sigma_v)}$.  Denote by $\sigma$ the intersection point $\rL_{(o,\sigma_v]}\cap \rL^+_{\pi_v,\pi_u}$.  From the construction, we know that  $u$ is coplanar to  $v$, $v_\sigma^+$ and $O$. Indeed,  it is lying in the  planar cone  bounded by the two rays $\rL^+_{O,v}$ and $\rL^+_{O, v^+_{\sigma}}$. The condition $ x_0 \leq r$ or $z_0 \leq s$ will ensure that  $u\in \bP^\circ_{Ov v_{\sigma}^+}$.  

Moreover, when $x_0\leq r, z_0\geq s$ or $x_0\geq r, z_0\leq s$, we have $(u-v)^2\geq (g-1)(y_0-c)^2 >0$. When $x_0\leq r$ and $z_0\leq s$, then we have  $$(u-v)^2\geq \frac{(c-y_0)^2}{c^2}v^2>0,$$
 by \eqref{eq:slope-in}.

\subsubsection*{Step 2}  Set $v(F_1)=(r',c', s')$ and $v(F_2)=(-r', mr-c', s-\tilde s-s')$ with $0<c'<m r$ and $ r'>0$. 
We claim that 
\begin{equation}\label{eq:slope-ineq}
    \mu_H(F_1)\geq \frac{c}{r}\quad \text{and} \quad \mu_H(F_2) \leq \frac{c-mr}{r}.
\end{equation}
Firstly, we must have either $r'\leq r$ or $s'\leq s$. Otherwise, one will have
     $$v(F_1)^2< (g-1)c^2-r(s+1)\leq -2$$
or 
      $$(v(F_1)-v|_C)^2< (g-1)(mr-c)^2-r(\tilde{s}+1)\leq -2.$$
Both of them are impossible as $v(F_1)\in \Omega_{v|_C}^+(\rL_{(o,\sigma_v]})$. 

Now, suppose $\mu_H(F_1) < \frac{c}{r}$.  
Then we have $v(F_1)\in \bP^\circ_{o \sigma_v \pi_v} $ as $\phi_{\sigma_v}(F_1) \geq \phi_{\sigma_v}(v)$.
According to Step 1,  we get  $$v(F_1)\in \Omega_v^+(\rL_{(o,\sigma_v]})$$ which contradicts to the assumption  $\Omega_{v}^+ (\rL_{(o, \sigma_v]})\cap \Halg(X)=\emptyset$.      Similarly,  we have  $\mu_H(F_2) \leq \frac{c-mr}{r} $ as  there is no integer point in $\Omega_{v(-m)}^+(\rL_{(o,\sigma_v]})$.  
This proves the claim.  As a consequence, we get  $$\frac{mr'}{r}=\mu_H(F_1)-\mu_H(F_2)\geq \frac{c}{r}-\frac{c-mr}{r}=m$$ which implies $r'\leq r$.

\subsubsection*{Step 3} Let $(F_1)_{\min}$ be the last $\mu_H$-HN factor of $F_1$, hence also of $F_1/T$. 
According to  \cite[Proposition~14.2]{Bridgeland2008}, for $\sigma$ sufficiently close to $o$, we always have 
\begin{itemize}
    \item $(F_1)_{\min}$ is $\sigma$-semistable,
    \item  $v(G)$ is proportional to $v((F_1)_{\min})$ for any $\sigma$-stable factor $G$ of $(F_1)_{\min}$.
\end{itemize}

  As $(F_1)_{\min} $ is a quotient sheaf of $F_1$, it is also a quotient of $F_1$ in $\Coh^{\beta=0}(X)$. 
Since $F_1$ is $\sigma_1$-semistable, we have $$\phi_{\sigma_1}(F_1) \leq \phi_{\sigma_1}((F_1)_{\min}).$$
Combined with the fact $\mu_H (F_1) \geq \mu_H((F_1)_{\min})$, we have $\pi_{G} = \pi_{(F_1)_{\min} } \in \bP_{o\sigma_1 \pi_{F_1}}$.
As the triangle $\bP_{o\sigma_1 \pi_{F_1} }$ is lying below the ray $\rL^+_{\sigma_v,\pi_v}$, we get $\pi_G \in \bP_{o\sigma_v \pi_v}^\circ$ if $\mu_H(G) < \frac{c}{r}$. 
Note that $\rk(G)\leq \rk(F_1)=r$. We must have  $\mu_H(G)\geq \frac{c}{r}$ otherwise one will get $\pi_G\in \Omega_v^+(\rL_{(o,\sigma_v]})$ by the same argument in Step 2. 
It follows that 
$$ \mu_H^-(F_1/T) =\mu_H((F_1)_{\min}) = \mu_H(G) \geq \frac{c}{r}.$$
A similar argument shows $\mu_H^+(\rH^{-1}(F_2)) \leq \frac{c-mr}{r}$. This finishes the proof. 
\end{proof}

\subsection*{HN-polygon} 
 Let $\sigma_{\alpha,0}$ be a stability condition with $\alpha$ close to $\sqrt{2/H^2}$.  
 By \cite[Proposition~3.4]{Feyzbakhsh2020}, for fixed $E$,  the HN filtration of $\sigma_{\alpha,0}$ will stay the same for $\sqrt{2/H^2}+\epsilon >\alpha >\sqrt{2/H^2}$.
Denote by $\overline\sigma$ the limit of $\sigma_{\alpha, 0}$. The `stability function' can be written as \[
\overline{Z}(E) = r-s + c\sqrt{-1} .
\]  
if $v(E)=(r,c,s)$.  Let $\bP_{i_*F}$ be the HN polygon\footnote{Here our definition of HN polygon is slightly different from \cite[Definition 3.3]{Feyzbakhsh2020}. We drop off the part on the right hand side of the line segment $\rL_{[0,\overline{Z}(i_*F)]}$.} for $i_*F$ with respect to $\overline \sigma$. For $E\in \bM(v)$, we have $\bP_{i_*(E|_C)}=\bP_{0z_1z_2}$, where
\[
z_1=r-s+c\sqrt{-1}  \quad \text{and} \quad z_2=m(g-1)(mr-2c)+mr\sqrt{-1}.
\]
As the polygon $\bP_{i_\ast(E|_C)}$ only depends  on $v$,  we may simply write it as $\bP_v$.

\begin{theorem} \label{thm_polygons}
For any $F\in \BN_C(v)$,  we have $\bP_{i_*F} \subseteq \bP_v$.  Moreover, they coincide if and only if $F=E|_C$ for some $E\in \bM(v)$.
\end{theorem}

\begin{proof} 
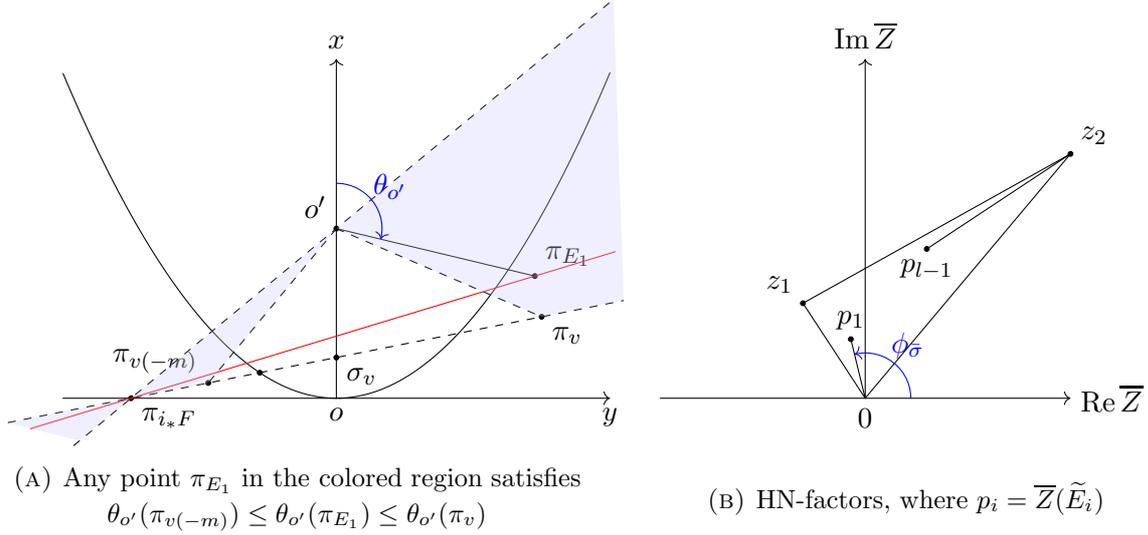
\begin{figure}
\centering
\begin{subfigure}{0.45\textwidth}
    \centering
\begin{tikzpicture}[domain=-4:4,samples = 1000,scale=0.9]
    \coordinate (O) at (0,0);
	\coordinate (o') at (0,2.5);
	\draw (O) node [below] {$o$};
	\filldraw [black] (o') circle [radius=1pt] node [above left] {$o'$};
	\draw [->] (-4,0) -- (4,0) node [below]{$y$};
	\draw[->] (O) -- (0,5) node [above]{$x$};
	\path [name path = x-axis] (0,0) -- (0,5);
	\path [name path = parabola] plot (\x,{0.3*(\x)^2}) ;
	\draw plot (-\x,{0.3*(\x)^2});
	\coordinate (E) at (3,1.2);
	\coordinate (F) at (-3,0);
    \filldraw [black] (E) circle [radius=1pt] node [below right] {$\pi_v$};
	\filldraw [black] (F) circle [radius=1pt] node [below right] {$\pi_{i_*F}$};
	\path [name path = link, add = 0.3 and 0.2, draw, dashed] (F) to (E);
	\path[draw, dashed] (o') to (E);
	\coordinate (E_1) at (2.9,1.8);
	\filldraw [black] (E_1) circle [radius=1pt] node [above right] {$\pi_{E_1}$};
	\path[add = 0.12 and 0, draw, dashed] (F) to (4,2.5/3*7);
	\path[add = 0.25 and 0.2, draw, red] (F) to (E_1);
    \draw[] (o') -- (E_1);
    \fill [fill = blue!20!white,opacity=0.3]  ($(o')!-1.35!(F)$) --(o')--(E) -- ($(E)!-0.2!(F)$) --cycle;
    \fill [fill = blue!20!white,opacity=0.3] (F) -- ($ (F)!-.3!(E) $)-- ($ (F)!-.25!(o') $)--cycle;
	\filldraw  [name intersections={of= link and parabola, by=E''},black] (E'') circle [radius=1pt] node [] {};
	\coordinate (E') at ($ (F)!.6!(E'') $);
	\filldraw [black] (E') circle [radius=1pt] node [above left] {$\pi_{v(-m)}$};
	\path[name path = o'E', draw,dashed] (o') to (E');
	\fill [fill = blue!20!white,opacity=0.3] (F) -- (E') -- (o') -- cycle;
	\filldraw  [name intersections={of= link and x-axis, by=sigma},black] (sigma) circle [radius=1pt] node [below right] {$\sigma_v$};
	\coordinate (x) at (0,5);
	\pic["$\theta_{o'}$", draw, blue, <-, angle eccentricity=1.5, angle radius=0.6cm]
    {angle=E_1--o'--x};
\end{tikzpicture}
\caption{Any point $\pi_{E_1}$ in the colored region satisfies $\theta_{o'}(\pi_{v(-m)}) \leq \theta_{o'}(\pi_{E_1}) \leq \theta_{o'}(\pi_v)$}\label{fig:angle_1}
\end{subfigure}\quad
\begin{subfigure}{0.45\textwidth}
\centering
\begin{tikzpicture}[domain=-3:3,samples = 1000,scale=0.9]
    \coordinate (O) at (0,0);
	\draw (O) node [below] {$0$};
	\draw [->] (-3,0) -- (3,0) node [right]{$\re \overline{Z}$};
	\draw[->] (0,0) -- (0,5) node [above]{$\im \overline{Z}$};
	\path [name path = y-axis] (0,-1) -- (0,5);
	\coordinate (E) at (0.7*-1.3,0.7*2);
    \filldraw [black] (E) circle [radius=1pt] node [above left] {$z_1$};
    \coordinate (F) at (1.2*2.5,1.2*3);
    \filldraw [black] (F) circle [radius=1pt] node [above right] {$z_2$};
    \draw[] (O) -- (E) -- (F)--cycle;
    \coordinate (p1) at (0.3*-0.7,0.3*2.9);
    \coordinate (p2) at (0.9, 2.2);
    \filldraw [black] (p1) circle [radius=1pt] node [above] {$p_1$};
    \filldraw [black] (p2) circle [radius=1pt] node [below] {$p_{l-1}$};
    \draw[] (O)-- (p1) (p2) -- (F);
	\coordinate (Re) at (1,0);
	\pic["$\phi_{\barsigma}$", draw, blue, ->, angle eccentricity=1.5, angle radius=0.6cm]
    {angle=Re--O--p1};
\end{tikzpicture}
\caption{HN-factors, where $p_i=\overline{Z}(\widetilde{E}_i)$}\label{fig:angle_2}
\end{subfigure}
\caption{The angle $\theta_{o'}$ is equal to  the angle $\phi_{\barsigma}$}
\end{figure}
When $v^2=0$, this is  essentially proved in \cite[Lemma 4.3]{Feyzbakhsh2020}. Let us give a slightly different argument which also works for $v^2>0$. Suppose the HN-filtration of $i_*F$ for $\bar{\sigma}=(\Coh^{\beta=0}(X), \overline{Z})$ is given by
\begin{equation}\label{eq:HN}
 0 = \widetilde{E}_0 \subset \widetilde{E}_1 \subset \dots \subset  \widetilde{E}_{l-1} \subset \widetilde{E}_l = i_* F
\end{equation}
with $E_i \defeq \widetilde{E}_i / \widetilde{E}_{i-1} $ the semistable HN-factors.  To show $\bP_{i_\ast F}\subseteq \bP_v$,  it suffices to show that \begin{equation} \label{eq_angle}
\phi_{\barsigma}(v)\geq \phi_{\barsigma}({E_1}) \quad \text{and} \quad \phi_{\barsigma}({E_l}) \geq  \phi_{\barsigma}({v(-m)})
\end{equation} 
(see \autoref{fig:angle_2}), since $ \bP_{i_*F}$ is convex.
According to the proof of Proposition~\ref{prop_phasereading}, for any object in $\Coh^0(X) $, the angle $\phi_{\barsigma}$ in \autoref{fig:angle_2} is an increasing function
\footnote{They are actually equal in this  case, as
$\cot \theta_{o'} = \frac{y-1}{x}=\frac{r-s}{c}=\frac{\re \overline Z}{\im \overline Z} = \cot \phi_{\barsigma}.$}
of the angle $\theta_{o'}$ in \autoref{fig:angle_1}.
Therefore, it is equivalent to show
\begin{equation} \label{eq_angle'}
\theta_{o'}(\pi_v)\geq \theta_{o'}(\pi_{E_1}) \quad \text{and} \quad \theta_{o'}(\pi_{E_l}) \geq  \theta_{o'}(\pi_{v(-m)})
\end{equation} 
in \autoref{fig:angle_1}.

To prove \eqref{eq_angle'}, consider the sequence
\begin{equation}\label{eq:ses}
    0 \to \widetilde{E}_n \xrightarrow{f_n} i_* F \to \cok(f_n) \to 0
\end{equation}
for each $\widetilde{E}_n$. 
Since the first wall is not below $\rL_{\pi_{i_* F}, \pi_v}$, we have $\phi_{\sigma_v}(\widetilde{E}_n)\leq  \phi_{\sigma_v}(i_*F)$. 
As $\phi_{\bar\sigma}(\widetilde{E}_n) \geq \phi_{\bar\sigma}(i_*F)$,
there exists some stability condition $\sigma \in \rL_{(o', \sigma_v]}$ such that the  objects in \eqref{eq:ses} have the same $\sigma$-phase. As a consequence, we have
\begin{equation}\label{eq_filtration pr}
\pi_{\widetilde E_i} \in \bigcup_{\sigma\in\rL_{[\sigma_v, o')}} \rL_{\pi_{i_*F, \sigma}}.    
\end{equation}
Take $n=1$ and set $v(E_1) = (r',c',s')$. 
We claim that $\pi_{E_1} \notin \bP_{o' \pi_v \sigma_v} \setminus \bigbrace{\pi_v}$ which yields $ \theta_{o'}(\pi_{E_1}) \leq \theta_{o'}(\pi_v)$. 
This can be proved by cases as follows: 
\begin{itemize}
    \item[Case (1)] If $v^2 =0$,  $\pi_{E_1} \notin  \bP_{o' \pi_v \sigma_v}\setminus\{ \pi_v\}$ automatically holds. 
    This is because  $\bP_{o' \pi_v \sigma_v}\setminus\{ \pi_v,o'\} \subseteq V(X)$ and \eqref{eq_def}. 
    
    \item[Case (2)] If $v^2>0$ and  $r' \leq r$ or $s' \leq s$, as $E_1$  is  $\bar{\sigma}$-semistable, we may assume $v(E_1)^2 \geq -2$ otherwise we may replace $E_1$ by its first JH-factor. According to the Step 1 in Lemma \ref{lem_factors},  we have $$v(E_1) \in \Omega_v^+(\rL_{(o,\,o')}). $$  
which contradicts to the assumption $\Omega_v^+(\rL_{(o,\,o')})\cap \Halg(X)=\emptyset$. 


    \item[Case (3)] If $r' > r$ and $s'>s$,  we claim that  $r' < r+1$. 
    Choose a stability condition $\sigma \in \rL_{[\sigma_v, o')}$ such that $\phi_{\sigma}(i_*F) = \phi_{\sigma}(E_1)$.
    Then $v(E_1)\notin\{ O, v|_c\}$ is lying in the triangle $\bP_{Ov|_c (v|_c)^+_{\sigma}}$. This means  we  have $0< c' < mr$ and
    \[
    g(c')^2 - r's' \geq 0,   \quad g(c'-mr)^2 -r'(s'+(g-1)m(mr-2c)) \geq 0
    \]
  After reduction, we know that  $r'<r+1$ as $g c^2 -(r+1)s \leq 0 $ and $ g (c-mr)^2 -(r+1)\tilde{s} \leq 0$ by \eqref{eq:inj-cond2}.
\end{itemize}

Similarly, take $n=l-1$ and use the $\barsigma$-semistability of $\cok(f_{l-1})=E_l$, one can prove the second inequality of \eqref{eq_angle'}.

Finally, if $\bP_{i_*F} = \bP_v$, the first wall will coincide with the line $\rL_{\pi_{v(-m)},\pi_v}$ and the last assertion follows from Theorem \ref{thm_1st wall}.
\end{proof}

\begin{remark}
The discussion above can be much more simplified if the following is true: for any $\sigma$ on a wall of $i_*F$, there exists a JH-filtration of $i_*F$ which is convex (i.e., the polygon with vertices $v(\widetilde E_i)$ is convex in the plane of $\bP_{O v v^+_\sigma}$).
\end{remark}

Now we provide a numerical criterion for verifying $\bP_v=\bP_{i_\ast F}$ via Euclidean geometry.  The key ingredient is the upper bound on the number of global sections of an object  $E\in \rD^b(X)$ established by Feyzbakhsh in   \cite{Feyzbakhsh2020,Feyzbakhsh2020+}. Recall that for any $x,y \in \ZZ$,  there is a function
\begin{equation*}
    \ell(x+\sqrt{-1} y)\defeq \sqrt{x^2 +2H^2 y^2 +4 (\gcd(x,y))^2}    
\end{equation*}
and one can define $\ell(E)\defeq \sum_i \ell(\overline Z(E_i))$ where $E_i$'s are the $\overline{\sigma}$-semistable factors of $E$. Moreover, we have a metric function  given by 
\begin{equation*}
    \|x+\sqrt{-1} y\| \defeq \sqrt{x^2 + (2H^2+4) y^2}   
\end{equation*}
and we set  $\|E\|\defeq \sum_i \|\overline Z(E_i)\|$. Clearly, one has $\|E\| \geq \ell(E)$  once the $y$-coordinates are non-zero. 

\begin{proposition}\cite[Proposition~3.3 and Remark~3.4]{Feyzbakhsh2020+}\label{prop_upper bound} 
Suppose $E\in \Coh^0(X)$ which has no subobject $F\subseteq E$ in $\Coh^0(X)$ with $c_1(F)=0$, we have
\begin{equation}\label{eq_fine upper bound}
    h^0(X,E)\leq \sum_i \left\lfloor \frac{\ell(E_i) + \chi(E_i)}{2} \right\rfloor = \sum_i \left\lfloor \frac{\ell(E_i)-\re \overline Z(E_i)}{2}\right\rfloor,
\end{equation}
where $E_i$'s are semistable factors with respect to $\overline \sigma$. In particular,
\begin{equation}\label{eq_coarse upper bound}
    h^0(X,E)\leq \frac{\|E\|+\chi(E)}{2}.
\end{equation}
\end{proposition}

Following \cite{Feyzbakhsh2020}, we can give a criterion for the  surjectivity of $\psi$.
\begin{theorem}\label{thm:surj}
With the notation as in \S 5.2. 
 Let $z_1^{+1}=r-s+1+c\sqrt{-1}$, $z'_1=r-s-\frac{r-s}{c}+(c-1)\sqrt{-1}$ and $z'_2=r-s-\frac{r-\gamma^2s}{\gamma c}+(c+1)\sqrt{-1}$ where $\gamma=\frac{mr}{c}-1$.   Assume that  
\begin{enumerate}
    \item $\frac{s-r}{c}+\frac{s-r-\chi}{m r-c}\geq 2$
    \item $ \|z_1- z'_1\|-\|z'_1 -z_1^{+1}\|+ \|z_1 -z'_2\|-\|z'_2 -z_1^{+1}\|\geq  \frac{2c^2}{r+s}+\frac{2(mr-c)^2}{r+s-\chi} $
\end{enumerate}
where $\chi=\chi(i_\ast F) = m(g-1)(2c-mr)$. Then the restriction map $\psi$ will be surjective. 
\end{theorem}

\begin{proof}
 Suppose we have $\bP_v\neq \bP_{i_\ast F}$ for some $F\in\BN_C(v)$. By Proposition \ref{prop_upper bound} and the convexity, we have
\begin{equation*}
    r+s \leq h^0(C, F)= h^0(X,i_*F) \leq \frac{\|i_* F\|+\chi}{2} \leq \frac{\hbar+\chi}{2}.
\end{equation*}
where   $ \hbar=\sqrt{(r+s-\chi)^2+4(mr-c)^2}+ \sqrt{(r+s)^2+4c^2}.$ Then we get 
\begin{equation}\label{eq:surjective}
    \frac{\hbar + \chi}{2} - (r+s)\geq   \frac{\hbar + \chi}{2} - \frac{\|i_\ast F\|+\chi}{2} =\frac{\hbar- \|i_\ast F\|}{2}.
\end{equation}
However, note that the polygon $\bP_{0z'_1z_1^{+1}z_2'z_2}$ is convex under the assumption (i), we have\begin{equation}
    \begin{aligned}
\hbar-  \|i_\ast F\| &\geq \|z_1 -z'_1\|-\|z'_1- z_1^{+1}\|+ \|z_1 -z'_2\|-\|z'_2- z_1^{+1}\|.
    \end{aligned}
\end{equation}
Combined with assumption (ii), we get
\begin{align}
{\hbar+ \chi} - 2(r+s) &= \sqrt{(r+s)^2+4c^2}-(r+s)        \nonumber\\
& \mathrel{\phantom{=}}+ \sqrt{(r+s-\chi)^2+4(mr-c)^2}-(r+s-\chi)             \nonumber \\
&< \frac{2c^2}{r+s}+\frac{2(mr-c)^2}{r+s-\chi}.\label{eq_maximal difference} \\ & \leq  \hbar-  \|i_\ast F\|
\end{align}
which contradicts to \eqref{eq:surjective}.  This proves the assertion.

\end{proof}

As an application,  we get an explicit criterion for $\psi$ being surjective for  $v^2\geq 0$. 
\begin{corollary}\label{cor:surj2}
The  restriction map $\psi:\bM(v)\to \BN_C(v)$ is bijective if we further have  $$ g\geq 4r^2+1.$$
\end{corollary}
\begin{proof}
 As $r>1+\frac{v^2}{2}$, we have  $$ s=\lfloor \frac{(g-1)c^2+1}{r}\rfloor \geq 4rc^2.$$ This gives
\begin{equation}\label{eq:surj1}
\begin{aligned}
s-r&\geq 4rc^2-r\geq 3c\\
     s-r-\chi & =\frac{(g-1)(mr-c)^2}{r}-\frac{v^2}{2r}-r
     \geq 4r(mr-c)^2-r \geq  3(mr-c),
\end{aligned}
\end{equation}
as $mr-c>0$.  Moreover, one can compute that  
 $$\frac{2s-2r-c}{2s+2r+c}  \geq  \frac{8rc^2-2r-c}{8rc^2+2r+c} \geq \frac{6r-1}{10r+1}>  \frac{1}{r}  \geq \frac{4c^2}{s}.  $$
It follows that \begin{align*}
    \|z_1- z'_1\|-\|z'_1 -z_1^{+1}\| &= \sqrt{(\frac{s-r}{c})^2+4g}-\sqrt{(\frac{s-r}{c}-1)^2+4g} \\
    &> \frac{\frac{s-r}{c}-\frac{1}{2}}{\sqrt{(\frac{s+r}{c})^2+4(1+\frac{n}{c^2})}} \\
    &> \frac{2s-2r-c}{2s+2r+c}> \frac{2c^2}{s}+\frac{2c^2}{r+s}.
\end{align*}
and $\|z_1 -z'_2\|-\|z'_2- z_1^{+1}\| =  \sqrt{(\frac{s-r-\chi}{mr-c})^2+4g}  - \sqrt{(\frac{s-r-\chi}{mr-c}-1)^2+4g} \geq 0$. The assertion can be concluded from Theorem \ref{thm:surj}. 
\end{proof}

\begin{remark}\label{rmk:surj}
For $g$ sufficiently large, it is not hard to find Mukai vectors satisfying the conditions in Theorem \ref{thm:surj}. For instance, when $v^2=0$ and $g>84$, the Mukai vectors given in \cite{Feyzbakhsh2020} and \cite{Feyzbakhsh2020+} will automatically satisfy the conditions for any $m\geq 1$. However, when $g$ is small, it becomes impossible to find such Mukai vectors.
\end{remark}

\section{Surjectivity for special Mukai vectors} \label{sec:surj2}

According to Remark \ref{rmk:surj},  Theorem \ref{thm:surj} does not work well for small $g$. 
In this section, we develop a way to improve the estimate in \S \ref{sec:surj} for special Mukai vectors of square zero.  Let us first introduce the sharpness of the polygon $\bP_v$.
\begin{definition} 
Denote by $z_1^{+d}$ the point $r-s+d+c\sqrt{-1}$. 
Let $z_1', z_2'$ be the points as in the Theorem \ref{thm:surj}.
We say the polygon $\bP_v$ is $d$-sharp if for any $\bP_{i_*F}\neq \bP_v$, one of the following is true:
\begin{enumerate}
    \item $\bP_{i_*F}$ is contained in the polygon $\bP_{0z'_1z_1^{+d} z_2' z_2}$.
    \item $z_1^{+j}$ is a vertex of $\bP_{i_*F}$ for some $1 \leq j \leq d-1$.
\end{enumerate}
\end{definition}

There is a simple numerical criterion for the $d$-sharpness of $\bP_v$.
\begin{lemma}\label{lem_sharp}
With the notations as before, suppose that
\begin{equation}\label{eq:k-sharp}
\frac{s-r}{c}+\frac{\gamma^2 s-r}{\gamma c}\geq 2d
\end{equation}
where $\gamma=\frac{mr}{c}-1$, the polygon $\bP_{v}$ will be $d$-sharp.
\end{lemma}

\begin{proof}
From the definition of two polygons, one observes that the interior of $\bP_{v}-\bP_{0z'_1z_1^{+d} z_2' z_2}$  only contains $z_1^{+j} ~ (1\leq j\leq d-1)$ as integer points. If $\bP_{0z'_1z_1^{+d} z_2' z_2}$ is convex, then either $\bP_{i_\ast F}$ is contained in $\bP_{0z'_1z_1^{+d} z_2' z_2}$  or $z_1^{+j}$ is a vertex of $\bP_{i_\ast F}$.   A little writing reveals the convexity of this polygon literally means \eqref{eq:k-sharp}.
\end{proof}

The following is an enhancement of Theorem \ref{thm:surj} for special Mukai vectors. 

\begin{theorem}\label{thm_surj2}
Suppose   $g\geq 3$. Let $v=(g-1,k,k^2)\in \rH_{\alg}^\ast(X)$ be a primitive Mukai vector with $\gcd(g-1,k)=1$. Assume that   $(m,k)$ satisfies the following three conditions
\begin{itemize}
    \item $ g < 2k$ and $g\neq k$
    \item $ g < 2(m(g-1) - k)$ and $g\neq m(g-1)-k$
    \item either $k\nmid g+1$ or $m(g-1)-k \nmid g+1$
\end{itemize}
Then the restriction map $\psi:\bM(v)\to \BN_C(v)$ is surjective. 
\end{theorem}

\begin{proof}
By Lemma \ref{lem_sharp},  if there is $\bP_{i_\ast F}\neq \bP_v$ for some $F$, the polygon $\bP_{v}$ will be at least $3$-sharp. Therefore, one of the following is true:
\begin{enumerate}
\item $\bP_{i_*F}$ is contained in the polygon $\bP_{0z'_1z_1^{+3} z_2' z_2}$. 
    \item $z_1^{+1}=g-k^2+k\sqrt{-1}$ is a vertex of $\bP_{i_*F}$ \label{item_+1}
    \item $z_1^{+2}=g+1-k^2+k\sqrt{-1}$ is a vertex of $\bP_{i_*F}$ \label{item_+2}
\end{enumerate}
We will analyse them by cases. Let us first show that case (i) is impossible if $(g,k,m)\neq (5,3,3)$. 
By \eqref{eq:surjective}, it suffices to show that 
\begin{equation}\label{eq:est-ineq}
\hbar-  \|i_\ast F\| >\hbar+\chi-2(g-1+k^2).
\end{equation}
when $\bP_{i_*F}\subseteq \bP_{0z'_1z_1^{+3} z_2' z_2}$.   Set $\tilde{k} = m(g-1)-k$.  As in the proof of Theorem \ref{thm:surj},  from the convexity and a direct computation, one can get 
\begin{equation}\label{eq:est-3}
\begin{aligned}
    \hbar-  \|i_\ast F\|& > \|z_1 -z'_1\|-\|z'_1 -z_1^{+3}\|+ \|z_1 -z'_2\|-\|z'_2- z_1^{+3}\|\\ 
    & = \sqrt{(\frac{k^2-g+1}{k})^2+4g} - \sqrt{(\frac{k^2-g+1}{k}-3)^2+4g} \\    & \mathrel{\phantom{>}} + \sqrt{(\frac{\tilde k^2-g+1}{\tilde k})^2+4g}  - \sqrt{(\frac{\tilde k^2-g+1}{\tilde k}-3)^2+4g}\\ &  \geq  \frac{4k^2}{\sqrt{(g-1+k^2)^2+4k^2}+(g-1+k^2)}  +\frac{4\tilde k^2}{\sqrt{(g-1+\tilde k^2)^2+4\tilde k^2}+(g-1+\tilde k^2)}  \\    &=\hbar+\chi-2(g-1+k^2) 
\end{aligned}
\end{equation}
whenever $(g,k,m)\notin \Big\{(5,3,m), (6,4,3), (8,5,2)\Big\}$ satisfies our assumption.  

In the case $(g, k,m)=(6,4,3)$, $(8,5,2)$, or $(5,3,m)$ with $m\geq 4$, though the inequality \eqref{eq:est-3} fails, one can give an improvement of the estimate \eqref{eq:est-3} by considering  the convex hull of integer points in $\bP_{0z_1'z_1^{+3}z_2'z_2}$. In those cases, the convex hull is a convex polygon  with vertices $z_1, z_1^{+3}, z_1', z_2',$ and $z_3$,  where $z_3$ is given as below:
\begin{itemize}
\item  $(g, k,m)=(5,3,m)$, $z_3=-3+2\sqrt{-1}$;
\item  $(g, k,m)=(6,4,3)$, $z_3=-8+3\sqrt{-1}$;
\item  $(g, k,m)=(8,5,2)$, $z_3=-14+4\sqrt{-1}$.
\end{itemize}
Then one can get
\begin{equation}
     \hbar-  \|i_\ast F\| > \|z_1\|-\|z_3\|- \|z_3-z_1^{+3}\|+ \|z_1 -z'_2\|-\|z'_2- z_1^{+3}\|
\end{equation}
A computer calculation of their values shows that \eqref{eq:est-ineq} still holds.

In case (ii) and (iii), if $z_1^{+1}$  or $z_1^{+2}$ is a vertex of  $\bP_{i_*F}$, there exists $\widetilde{E}_j \subset i_*F$ in the HN-filtration \eqref{eq:HN} such that $\overline Z(\widetilde{E}_j)=z_1^{+1}$ (respectively, $z_1^{+2}$).  Then we have 
\begin{equation}\label{up-bound}
\begin{aligned}
   h^0(X,i_*F)  &\leq \sum_i \left\lfloor \frac{ l (E_i)+\chi(E_i)}{2}\right\rfloor \nonumber \\
    &\leq  \sum_{i\leq j} (\left\lfloor \frac{l (E_i)+\chi(E_i)}{2}\right\rfloor -\frac{\chi(E_i)}{2}) + \sum_{i > j}  (\left\lfloor \frac{l (E_i)+\chi(E_i)}{2}\right\rfloor-\frac{\chi(E_i)}{2} ) +\frac{\chi}{2}\\
    &  \leq \frac{\hbar+\chi}{2} \nonumber
\end{aligned}    
\end{equation}
For simplicity, we may use $\frac{\hbar_1}{2}$ and $\frac{\hbar_2}{2}$ to denote the first two terms in the second row. As $h^0(X,i_\ast F)\geq g-1+k^2$, following the argument in Theorem~\ref{thm:surj}, it suffices to prove the inequality 
\begin{equation*}
    \frac{\hbar + \chi}{2} - (g-1+k^2)< \frac{\hbar-\hbar_1-\hbar_2}{2}, 
\end{equation*}
 or equivalently,
\begin{equation}\label{eq:surjective++}
   \hbar_1+\hbar_2 < 2(g-1+k^2) -\chi=\|z_1^{+1}\| +\|z_1^{+1} -z_2\|-2.
\end{equation}
For case ii), a direct computation  shows
    \begin{align}
    \|z_1^{+1}\|-\hbar_1& = \sum_{i\leq j} \|E_i\|-\hbar_1-(\sum_{i\leq j}\|E_i\|-\|z_1^{+1}\|) \nonumber\\  
         & \geq  \sum\limits_{i\leq j}  (\|E_i\|- \ell(E_i))-(\sum_{i\leq j}\|E_i\|-\|z_1^{+1}\|) \nonumber \\  &\geq  (\|E_1\|- \ell(E_1))-(\sum_{i\leq j}\|E_i\|-\|z_1^{+1}\|)\nonumber \\ & \geq  \frac{3}{\sqrt{k^2+4g-3}}-(\sum_{i\leq j}\|E_i\|-\|z_1^{+1}\|) \label{ie1}\\ &\geq \frac{3}{\sqrt{k^2+4g-3}}-(\|z'_1\| + \|z'_1-z_1^{+1}\|-\| z_1^{+1}\| ) \label{ie:conv}\\ &>0. \label{ie:dircom} 
    \end{align}
Let us explain why the inequality \eqref{ie1} holds.  Note that $\overline{Z}(E_1)=x+y\sqrt{-1}$ satisfies
\[k^2-g \leq -\frac{x}{y} \leq k^2-g+1.\]
Then  we have $-x< ky$ and $y \nmid x$ by our assumption $g-1\neq k$, $g\neq k$, and $g<2k$. This will give 
\begin{align*}
    \|E_1\|- \ell(E_1) &\geq \sqrt{x^2 +4g y^2}- \sqrt{x^2+4(g-1)y^2+y^2} \quad \textrm{(since $ y\nmid x$)}\\
    & \geq \frac{3y^2}{2 \sqrt{x^2+(4g-3)y^2}} \\
    & \geq \frac{3y^2}{2 \sqrt{k^2y^2+(4g-3)y^2}} \\
    & > \frac{3}{\sqrt{k^2+4g-3}} \quad \textrm{(since $y \geq 2$, which is a consequence of $ y\nmid x$)}.
\end{align*}
The inequality \eqref{ie:dircom} holds  because
\begin{align*}
    \|z_1^{+1} -z'_1 \|+\|z'_1 \| -\|z_1^{+1}\|&= \sqrt{(\frac{k^2-(g-1)}{k}-1)^2+4g} + (k-1)\sqrt{(\frac{k^2-(g-1)}{k})^2+4g } -(k^2 +g) \\
    &=\sqrt{(\frac{k^2-(g-1)}{k}-1)^2+4g}- \left(k-1+\frac{g+1}{k}\right)\\
    &\mathrel{\phantom{=}}+(k-1)\sqrt{(\frac{k^2-(g-1)}{k})^2+4g } -\left((k^2 +g)-(k-1+\frac{g+1}{k})\right)\\
    &< \frac{2g(k-1)}{k^3-k^2+(g+1)k} + \frac{-2 g (k-1)^2}{k^4-k^3+(g+1)k^2-(g+1)k}\\
    &= \frac{2g(k-1)}{(k^2+g+1)(k^2-k+g+1)} \\
    & \leq \frac{3}{\sqrt{k^2+4g-3}},
\end{align*}
when $k\geq \frac{g+1}{2}\geq 2$. 
Note that $ \|z_1^{+1}\|-\hbar_1=k^2+g- 2\sum\limits_{i\leq j} \left\lfloor \frac{l (E_i)+\chi(E_i)}{2}\right\rfloor-\sum\limits_{i\leq j} \chi(E_i)$ is an even number, this yields
$  \|z_1^{+1}\|-\hbar_1 \geq 2.$ 

Next,   recall that  $\widetilde{E}_l=i_\ast F$ in the HN filtration \eqref{eq:HN}, we can get
\begin{align*}
    \|z_1^{+1} -z_2\|-\hbar_2 & \geq  \sum\limits_{i >j}  (\|E_i\|- \ell(E_i))-(\sum_{i> j}\|E_i\|-\|z_1^{+1}- z_2\|) \\
    &\geq  (\|E_l\|- \ell(E_l))-(\|z_1^{+1}-z'_2\| + \|z'_2- z_2\|-\|z_1^{+1} - z_2\| ) \\
    &> \frac{3}{\sqrt{\tilde k^2+4g-3}} - \frac{2g(\tilde k-1)}{(\tilde k^2+g+1)(\tilde k^2-\tilde k+g+1)}>0
\end{align*}
as $\tilde{k}>\frac{g+1}{2}$. Combining them together, we can obtain \eqref{eq:surjective++}. 

For case (iii), if $k\nmid g+1$, we have 
\begin{align}
    \|z_1^{+1}\|-\hbar_1& =  \sum_{i\leq j}\|E_i\|-\hbar_1+\|z_1^{+1}\|-\sum_{i\leq j}\|E_i\| \nonumber \\ 
    &\geq \|E_1\|-\ell(E_1) +\|z_1^{+1}\|-\|z'_1\|-\|z'_1-z_1^{+2}\| \nonumber \\ 
    &\geq \frac{3}{\sqrt{k^2+4g-3}} + \|z_1^{+1}\|-\|z'_1\|-\|z'_1-z_1^{+2}\| \\ &>1 \label{ie:+2}
\end{align}
Here, the inequality \eqref{ie:+2} holds because
\begin{align*}
    \|z_1^{+2}-z'_1 \|+\|z'_1\| -\|z_1^{+1}\|
    &=\sqrt{(\frac{k^2-(g-1)}{k}-2)^2+4g}+(k-1)\sqrt{(\frac{k^2-(g-1)}{k})^2+4g} -(k^2+g)\\
    &=\sqrt{(\frac{k^2-(g-1)}{k}-2)^2+4g}-(k-2+\frac{g+1}{k})\\
    &\mathrel{\phantom{\geq}} +  (k-1)\sqrt{(\frac{k^2-(g-1)}{k})^2+4g}-(k^2-k+(g+1)-\frac{g+1}{k})-1 \\
    &\leq \frac{2 g (2 k-1)}{k \left(g+(k-1)^2\right)}-\frac{2 g (k-1)}{k \left(g+k^2+1\right)}-1\\
    &= \frac{2 g \left(k^2+2k+g-1\right)}{\left(k^2+g+1\right)\left(k^2-2k+g+1\right)} -1\\
    &< \frac{3}{\sqrt{k^2+4g-3}}-1.
\end{align*}
Note that $\|z_1^{+1}\|-\hbar_1$ is an odd number, this yields $\|z_1^{+1}\|-\hbar_1 \geq 3$. Similarly, one can get  
\begin{align*}
    \|z_1^{+1}- z_2\|-\hbar_2\geq 3
\end{align*}
under the assumption $m(g-1)-k \nmid g+1 $. Since both of them are at least positive under our assumption, we get \eqref{eq:surjective++} as well. This finishes  the proof for $(g,k,m)\neq (5,3,3)$.

For the remaining case $(g,k,m)=(5,3,3)$, we have to make use of the $4$-sharpness of $\bP_v$.  We just need to verify  $\bP_{i_*F}$ is not contained in  $\bP_{0z'_1z_1^{+4} z_2' z_2}$ and $z_1^{+3}=-2+3\sqrt{-1}$ is not a vertex of $\bP_{i_\ast F}$. As above, by using the convex hull of integer points in  $\bP_{0z'_1z_1^{+4} z_2' z_2}$, we have 
\begin{equation*}
\begin{aligned}
 \hbar-  \|i_\ast F\| \geq & \|z_1\|-\|-3+2\sqrt{-1}\|- \|-3+2\sqrt{-1}-z_1^{+4}\| + \|z_1 -z'_2\|-\|z'_2- z_1^{+4}\|\\ 
    &=-2 \sqrt{6}-\sqrt{89}+\sqrt{205}+\frac{\sqrt{7549}}{9}-\frac{\sqrt{3301}}{9} \\
    &>\hbar+\chi-2(g-1+k^2).
\end{aligned}
\end{equation*}
which show that $\bP_{i_*F}$ cannot lie in  $\bP_{0z'_1z_1^{+4} z_2' z_2}$.  Moreover, a similar estimate of $\|E_1\|-\ell(E_1)$ and $\|E_l\|-\ell(E_l)$ in (ii) and (iii) shows that $z_1^{+3}$ is not a vertex of $\bP_{i_\ast F}$. 
\end{proof}

\begin{remark}
One can also directly check the small genera cases by running  the computer program in \cite[Section 4]{Feyzbakhsh2020+} .
\end{remark}

\section{Surjectivity of the tangent map}\label{sec:iso}
 
In this section, we adapt Feyzbakhsh's approach to study the surjectivity of the tangent map and obtain a sufficient condition for the restriction map being an isomorphism.  
\begin{theorem}\label{thm:iso}
 Let $v=(r,c,s)\in\Halg(X)$ be a  Mukai vector satisfying \eqref{eq:inj-cond} and $\eqref{eq:inj-cond2}$.  The  morphism $$\psi:\bM(v)\to \BN_C(v)$$ is an isomorphism whenever the following conditions hold \begin{enumerate}
     \item $\psi$ is surjective;
     \item $h^0(X,E) = r+s$ for any $E \in \bM(v)$;
     \item there exists $\sigma \in  \rL_{(\pi_{v(-m)},\pi_{v_K})}\cap V(X)$ such that 
     \begin{equation}
    \Omega^+_{v(-m)}(\rL_{(o,\sigma]}) \cap \Halg(X) = \Omega^+_{v_K}(\rL_{(o,\sigma]}) \cap \Halg(X)=\emptyset, 
    \end{equation}
    where $v_K = (s,-c,r)$;
    \item $2s>v^2+2c^2$, or $2s>v^2+2$ and $\gcd(c,s)=1$.
 \end{enumerate} 
\end{theorem}

\begin{proof}
As $\psi$ is bijective, it suffices to show the tangent map $\rmd\psi$ is surjective. The argument is similar as   \cite[\S 6]{Feyzbakhsh2020}.  For the convenience of readers, we  sketch the proof as below. For any  $E\in \bM(v)$, the differential map $\rmd\psi : T_{[E]} \bM(v) \to T_{[E|_C]}\BN_C(v)$   at $[E]$ can be identified as the map 
$$ \rmd \psi: \Hom(E, E[1])\to \ker \big(\Hom_C(E|_C,E|_C[1]) \xlongrightarrow{\coho^0} \Hom(\coho^0(C,E|_C), \coho^1(C,E|_C) \big)$$
sending $(E\to E[1])$ to $(E|_C \to  E|_C[1]) $. 

Let $\xi \colon E|_C\to E|_C[1]$ be a tangent vector in $ T_{[E|_C]}\BN_C(v)$. 
Then  Feyzbakhsh has shown in \cite[\S 6]{Feyzbakhsh2020} that  there exist morphisms $\xi' $ and $\xi'' $ such that  the following  commutative diagram holds  \[
\begin{tikzcd}
E \arrow[r] \arrow[d, "\exists \xi'"', dashed] & i_* E|_C \arrow[r] \arrow[d, "i_*\xi "] & {E(-C)[1]} \arrow[d, "\exists \xi''", dashed] \\
{E[1]} \arrow[r]  \arrow[rr, bend right, "0"]                      & {i_*E|_C[1]} \arrow[r]            & {E(-C)[2]}                                   
\end{tikzcd}
\]
provided that
\begin{equation}\label{condition_tangent map}
     K_E = M[1] \quad \text{and} \quad \Hom_X(M,E(-C)[1])=0
    \end{equation}
where $K_E$ is the cone of the evaluation map $\cO_X^{h^0(X,E)} \rightarrow E \rightarrow K_E$ in $\dCat(X)$.
 Note that  $\rmd\psi(\xi') = \xi$,  we are therefore reduced to check   \eqref{condition_tangent map} holds for every $E$.

Note that   $v(K_E) = -v_K$ and $\pi_{v_K} = \pi_{K_E}$.  We can choose the stability condition $\sigma_1\in \rL_{( \pi_{v_K}, o')}$ sufficiently close to $o'$ and $\sigma_2 \in \rL_{(\pi_{K},o)} $ sufficiently close to $o$ so that  
$$\bP_{o\sigma_2 \sigma_1 o'}\setminus \{o,o'\} \subseteq\bP_{o\pi_{v_K} \infty} \cap \Gamma\subseteq  V(X), $$
see Figure \ref{fig:tangent map}.
As in the proof of Theorem \ref{thm_restriction map},
 we have $\cO_X$ and $E$ are $\sigma_1$-semistable of the  same phase.
 Then as the quotient of $E$ by $\cO_X^{h^0(X,E)}$, $K_E$  is also $\sigma_1$-semistable of the same $\sigma_1$-phase. 
Note that Lemma \ref{lem_no int} still holds if we exchange $r$ and $s$ in Mukai vector $v$. 
Then we get
\[
\Omega^+_{v_K}(\bP_{o\pi_{v_K} \infty} \cap \Gamma) \cap \Halg(X) = \emptyset.
\]  
Since $v_K$ is primitive, we have $ \Omega_{v_K}(\bP_{o\pi_{v_K} \infty} \cap \Gamma) \cap \Halg(X) = \emptyset$.
By Proposition \ref{prop:triangle rule+}, $v_K$ admits no strictly semistable stability conditions in  $\bP_{o\sigma_2 \sigma_1 o'}\setminus \{o,o'\}$.  Therefore, $K_E$ is stable for any $\tau \in \bP_{o\sigma_2 \sigma_1 o'}\setminus \{o,o'\}$.
This implies that $K_E$ is $\sigma_{\alpha,0}$-stable for $\alpha > \sqrt{\frac{2}{H^2}}$. 
By \cite[Lemma 6.18]{Macri2017},  we have  $\coho^{-1}(K_E)$ is a $\mu_H$-semistable torsion free sheaf and $\coho^0(K_E) $ is a torsion sheaf supported in dimension zero. 
So we can set $v(\coho^0(K_E)) = (0,0,a)$ and $v(\coho^{-1}(K_E)) = (s,-c,r+a)$ for some $a\geq 0$. By \cite[Lemma 3.1]{Feyzbakhsh2020}, we have 
\begin{equation}\label{eq:torsion zero}
-2c^2 \leq v(\coho^{-1}(K_E))^2 = v^2 - 2 s a.
\end{equation}
When $\gcd(c,s)=1$, we have $\coho^{-1}(K_E)$ is slope stable and 
$v(\coho^{-1}(K_E))^2\geq -2$.  Then by condition (iv), we have $a=0$ and $\coho^0(K_E)=0$.
So we obtain  $K_E =M[1]$, where $M = \coho^{-1}(K_E)$ is a $\mu_H$-semistable torsion free sheaf.

Since  $\Omega^+_{v_K}(\rL_{(o,\sigma]}) \cap \Halg(X)=\emptyset,$
$v(M)$ admits no strictly semistable condition in $\rL_{(o,\sigma]}$.
 It follows that  $M$ is $\sigma$-stable as it is $\sigma_2$-stable. Similarly, we have $E(-C)[1]$ is also $\sigma$-stable.
Since $M$ and $E(-C)[1]$ are $\sigma$-stable of the same phase, 
one must have $\Hom_X(M,E(-C)[1])=0$. This proves the assertion. 
\end{proof}

\begin{figure}[ht]
\centering
\begin{tikzpicture}[domain=-3.5:3.5,trim right = 5cm,samples = 100]
    \coordinate (O) at (0,0);
	\coordinate (o') at (0,2.8);
	\draw (O) node [below] {$o$};
	\filldraw [black] (o')  circle [radius=1pt] node [above right] {$o'$};
	\draw [->] (-4,0) -- (4,0) node [right]{$y$};
	\draw[->] (0,0) -- (0,5) node [above]{$x$};
	\path [name path = x-axis] (0,-1) -- (0,5);
	\path [name path = parabola] plot (\x,{0.4*(\x)^2}) ;
	\draw plot (\x,{0.4*(\x)^2}) node[right] {$x =\frac{H^2}{2}y^2$};
	
	\coordinate (E) at (2.4,1.6);
	\coordinate (E') at (-1.1,0.3);
    \filldraw [black] (E) circle [radius=1pt] node [below right] {$\pi_v$};
    \filldraw [black] (E') circle [radius=1pt] node [left] {$\pi_{v(-m)}$};
	\path[name path = KE, add= 1.7 and .5] (o') to (E);
	\path[name intersections={of= KE and parabola, by=K'}];
	\coordinate (K) at ($(K')!-0.03!(E)$);
 	\filldraw [black] (K) circle [radius=1pt] node [below left] {$\pi_{v_K}$};
 	\draw[densely dotted] (E) -- (K);
	\draw[densely dotted] (K) -- (O);
	\draw[densely dotted] (K) -- (E');
	\draw[densely dotted] (E') -- (O);
	\draw[densely dotted] (E) -- (O);
	
	\node[inner sep=0pt]  (sigma_1) at ($ (K)!.9!(o') $) {};
    \filldraw [black] (sigma_1) circle [radius=1pt] node [above] {$\sigma_1$};
    \coordinate (sigma_2) at ($(K)!0.9!(O)$);
    \filldraw [black] (sigma_2) circle [radius=0.5pt] node [right] {$\sigma_2$};
    \draw [blue,opacity=0.4] (sigma_1) -- (sigma_2);
    \filldraw [fill = blue!20!white,opacity=0.3]  (o') -- (O)-- (sigma_2) -- (sigma_1)  --cycle;    
    \coordinate (sigma) at ($(K)!0.5!(E')$);
    \filldraw [black] (sigma) circle [radius=1pt] node [below left] {$\sigma$};
    \draw[blue,opacity=0.6] (O)--(sigma);
\end{tikzpicture}
\caption{}
\label{fig:tangent map}
\end{figure}

As a first application, we obtain a numerical criterion for Mukai's program. 
\begin{theorem}
\label{new-region-K3} 
Assume $g>2$. Let $v=(r,c,ck)\in \Halg(X)$ be  a primitive Mukai vector with $v^2=0$ and $\gcd(r,c)=1$. Suppose that
\begin{center}
     $r \mid g-1$, $k\nmid g$, $0<k\leq 3g-3 $ and  $m>1+\frac{ck}{r(k-1)}$.
\end{center}   The restriction map $\psi:\bM(v)\to \BN_C(v)$ is an isomorphism if it is a surjective morphism. 
\end{theorem}
\begin{proof}
Let us check that the conditions (ii)-(iv) in Theorem \ref{thm:iso} are satisfied.  
By our assumption,  we know that $\gcd(r-s,c)=1$. According to \cite[Lemma 3.1]{Feyzbakhsh2020+},  one has $$h^0(X,E) \leq r+s,$$ which forces  $h^0(X,E)= r+s $ by \eqref{eq_global section X}.  This verifies the condition (ii).  

For the condition (iv), note that we have $s=\frac{c^2(g-1)}{r}\geq c^2$ where the equality holds when $r= g-1$.
If $r=g-1$, the inequality in \eqref{eq:torsion zero} will be equality. 
By \cite[lemma 3.1]{Feyzbakhsh2020}, we have $c\mid(g-1)$ which is a contradiction.
Thus we only need to verify the condition (iii). By Remark \ref{rmk:square zero},  it suffices to show  $$\bP_{o \pi_{v_K} \pi_{v(-m)}} \setminus \{\text{vertices}\}\subseteq V(X).$$  
To make the computation easier,  we  may consider the action  of tensoring the invertible sheaf $\cO_X(H)$ which sends  the triangle  $\bP_{o \pi_{v_K} \pi_{v(-m)}}$ to the triangle  $\bP_{op_1p_2}$, where $p_1=\pi_{v_K(1)}$ and $p_2=\pi_{v(1-m)}$.  Then it is equivalent to show  there are no projection of roots in $\bP_{op_1p_2}-\{\text{vertices}\}$.

Firstly,   we show that there is no projection of root on the two edges joining $o$.  By definition, we have 
\begin{center}
   $p_1=(\frac{kc}{(k-1)^2 r},\frac{c}{(k-1)r})$ and $p_2=(\frac{cr}{k((m-1)r-c)^2}, \frac{c}{k(c-(m-1)r)}).$
\end{center}
Then two open edges $\rL_{(o,p_1)}$ and $\rL_{(o,p_2)}$ do not contain any projection of roots by Observation \ref{obs-B}.  

Next, since $(m-1)r>\frac{ck}{(k-1)}>c$, we know that  $p_1$ is lying in the first quadrant while  projection $p_2$ is lying in the second quadrant. 
So the region $$\bP_{op_1p_2} \setminus (\rL_{[o,p_1]}\cup \rL_{[o,p_2]}) $$ is contained in the union of two trapezoidal regions  
$\bP_{o p_1 \infty}^\circ$, $\bP_{o p_2\infty}^\circ$ and the $x$-axis. 
As $r|(g-1)$ and $\gcd(r,c)=1$,  we have an inclusion $$ \bP_{o p_2 \infty}^\circ\subseteq V(X)$$
from  Lemma \ref{lem_no int} (i).
Moreover,  if there is a root  $\delta=(r',c',s')\in \rR(X)$ with $r'>0$ whose projection  $\pi_\delta$ is lying  in  $\bP_{o p_1 \infty}^\circ$,  one can  follow the computation in Lemma \ref{lem_no int} to get  inequalities 
\begin{equation} \label{eq:root in region++}
    k c' < (k-1) r'  < k c' +\frac{k}{(g-1)c'} 
\end{equation}
However, one can directly check that there are no such integers $(r',c',s')$ satisfying \eqref{eq:root in region++} under the assumption $k\leq g-1$ or $3<k\leq 3g-3$.   

It remains to show that $\bP_{op_1p_2} \cap \text{$y$-axis}\subseteq V(X)$. Note that $$ \rL_{[p_1, p_2]} \cap \text{$x$-axis}=( \frac{c/(k-1)}{(m-1)r-c},0)$$  
 which is below $o'$ . It follows that  $\bP_{op_1p_2} \cap \text{$y$-axis}\subseteq \rL_{(o, o')}\subseteq V(X)$.

\end{proof}

 Moreover, we can reconstruct hyperk\"ahler varieties as Brill--Noether locus for Mukai vectors given in Corollary \ref{cor:surj2}. 
\begin{theorem}\label{thm_isomorphism HK}
Under the assumptions in Corollary \ref{cor:surj2}, the restriction map $\psi:\bM(v)\to \BN_C(v)$ is an isomorphism.
\end{theorem}
\begin{proof}
As above, we only need to verify the conditions (ii)-(iv) in Theorem \ref{thm:iso}. We will check them one by one. 

\begin{enumerate}
    \item [(1)] Let us first verify that $h^0(X,E) =r+s$ for any $E\in \bM(v)$. 
By \cite[Proposition 3.1]{Feyzbakhsh2020+}, it suffices to show that \begin{equation} \label{eq:global section estimate+}
    \frac{\sqrt{(r-s)^2+ (2g+2)c^2}}{2} < \frac{r+s}{2} +1 
\end{equation}
After simplification, one can find that   \eqref{eq:global section estimate+} is equivalent to \[
\frac{(g+1)c^2-1}{2r+2}-1 <s.
\] This holds when $(g-1)c^2 -rs <r$ and $g> 4r^2+1$.
\vspace{.1cm}

    \item [(2)]For condition (ii), we claim that $$\rL_{(\pi_{v(-m)},\pi_{v_K})}\cap  \Gamma \neq \emptyset,$$
and hence  $\rL_{(\pi_{v(-m)},\pi_{v_K})}\cap V(X)\neq \emptyset$.   Let us write  $v(-m)=(r, \tilde{c},\tilde{s} )$ and  $v_K=(s,-c,r)$ with  $\tilde{c}=c-mr$ and  $ \tilde s=\lfloor\frac{(g-1)\tilde c^2 +1}{r}\rfloor~ $. 
Then we only  need to show that  the quadratic equation 
\begin{equation}\label{eq:intersection with gamma}
    g((1-t)t\tilde{c}+t(-c))^2 = ((1-t)r+ t s)((1-t)\tilde{s}+ t r)
\end{equation}
has  roots for $0<t<1$.   By calculating the discriminant of \eqref{eq:intersection with gamma}, we know it has a solution $t_0$ satisfying 
\begin{equation}\label{eq:bound}
    0<t_0 < \frac{\tilde{c}^2 g -r\tilde{s}}{s\tilde{s}+r^2 +2\tilde c cg+2(\tilde c^2 g-r\tilde s)}<1.
\end{equation}

\item[(3)]  Choose $\sigma\in  \rL_{(\pi_{v(-m)}, \pi_{v_K})} \cap \Gamma $, we first  verify that $$\Omega^+_{v(-m)}(\rL_{ (o,\sigma]}) \cap \Halg(X)= \emptyset.$$  
Suppose there is an integer point $p_0=(x_0,y_0,z_0) \in \Omega^+_{v(-m)}(\rL_{[ \sigma, o)})$. By Lemma \ref{lem:c-value estimation}, we have  $$\tilde{c}-1 < y_0< 0.$$
Moreover, one may observe that  $p_0$ is lying in a (closed) planer region enclosed by the  conic 
\begin{equation*}
 Q= \left\{ y = y_0 ,\;  (g-1)y^2+1= xz \right\}
    \end{equation*}
and two lines 
$$ L_1=\left\{y = y_0,\; z=\frac{y_0 \tilde s}{\tilde{c}}  \right\};~ L_2 =\left\{ (1-t)\frac{y_0}{\tilde{c}}(r,\tilde{c},\tilde{s} )-\frac{ty_0}{c}(s,-c,r) ,t\in \RR \right\}.$$
It has three vertices given by  the intersection points $L_1\cap L_2$, $L_1\cap Q$ and $L_2\cap Q$. 
This yields 
\begin{equation}\label{eq:easy-estimate}
\begin{aligned}
   \frac{y_0 r}{\tilde c} \leq x_0 \leq  (1-t') \frac{y_0r}{\tilde c}  -\frac{t'y_0 s}{c}
\end{aligned}
\end{equation}
where $t'$ is the smaller root of the quadratic equation $$(g-1)y_0^2+1= [(1-t) \frac{y_0r}{\tilde c}  -\frac{ty_0  s}{c} ] [(1-t) \frac{y_0\tilde s}{\tilde c}  -\frac{ty_0 r}{c}].$$  
Solving the equation, one can get  
\begin{equation}\label{eq:estimate-t}
\begin{aligned}
t' \leq 2 \frac{(g-1) y_0^2+1-\frac{r\tilde s y_0^2}{\tilde c^2}}{-\frac{\tilde c \tilde s s}{c} -2r\tilde s} 
\leq \frac{2c(y_0^2r+\tilde c^2)}{ (-\tilde c\tilde s s-2r c\tilde s)\tilde c^2}
\end{aligned}
\end{equation}
as $(g-1)\tilde c^2-r\tilde s< r$.  Plugging \eqref{eq:estimate-t} into \eqref{eq:easy-estimate}, we get 
\begin{equation*}
   \begin{aligned}
0< x_0  -\frac{y_0r}{\tilde c}&\leq 
\frac{2cy_0(y_0^2r+\tilde c^2)}{ (\tilde c\tilde s s+2r c\tilde s)\tilde c^2} (\frac{s\tilde c+rc}{c\tilde c})  
\\&= (\frac{y_0}{\tilde c})(\frac{y_0^2r+\tilde c^2}{\tilde s\tilde c^2}) (\frac{s\tilde c+rc}{s\tilde c+2rc})
\\ 
&\leq  \frac{3(r+1)}{\tilde s } 
\\ 
& < \frac{3r(r+1)}{(g-1)\tilde{c}^2-r} \\
&< -\frac{1}{\tilde c}
    \end{aligned}
\end{equation*}
where the last inequality holds because $g-1\geq 4r^2$. 
This  contradicts to $x_0,y_0\in \ZZ$.  A similar computation shows that $\Omega^+_{v_K}(\rL_{(o, \sigma]}) \cap \Halg(X)=\emptyset$ as well. 

\item[(4)] Condition (iv) holds since our assumption $g >4r^2+1 $  ensures that \[
2s> 2r-2 +2c^2 >v^2 + 2c^2.
\]
\end{enumerate}
\end{proof}

\begin{lemma}\label{lem:c-value estimation}
For any integer point $(x_0,y_0,z_0) \in \Omega^+_{v(-m)}(\rL_{(o,\sigma]})$ in  Theorem \ref{thm_isomorphism HK},  we have 
\[
 \tilde{c}-1 < y_0 <0
\]
\end{lemma}

\begin{proof}
Set $v_{t} = (1-t) v(-m) + t v_K$.  Let $0<t_0<t_1<1 $ be the  roots of equation \eqref{eq:intersection with gamma}. We set $$w =(x',y',z')= \rL_{O,v_{t_0}} \cap \rL_{v(-m), v(-m) + v_{t_1}}.$$  
Then  $\Omega^+_{v(-m)}(\rL_{(o,\sigma]})$ is contained in the  tetrahedron $ \bT_{Ov(-m)\varpi w}$ with $4$ vertices $O, v(-m), w $ and $\varpi = (r, \tilde{c}, \frac{g\tilde{c}^2}{r})$. This gives  $y'<y_0<0 $. Hence we only need to estimate the lower bound of $y'$. 

Set $v_{t_0}=(r_{t_0},c_{t_0}, s_{t_0})$, then we have 
$w=\frac{y'}{c_{t_0}}v_{t_0}\in \rL_{o,v_{t_0}}$. Note that $w - v(-m)\in \rL_{O, \, v_{t_1}}$ is lying on the hyperboloid $$ \Big\{(x,y,z)\in \RR^3|~g y^2 -x z =0\Big\}.$$  
Then one can see that $y'<\tilde c-1$  if  
\begin{equation}
\begin{aligned}
     &\frac{\tilde{c}-1}{c_{t_0}} v_{t_0} -v(-m) \in \Big\{(x,y,z)\in \RR^3|~g y^2 -x z <0\Big\},
\end{aligned}
\end{equation}
i.e. \begin{equation}\label{eq:ineq-hyper}
    \big( \frac{\tilde{c}-1}{c_{t_0} }r_{t_0} -r\big) \big(\frac{\tilde{c}-1}{c_{t_0} }s_{t_0}-\tilde{s} \big) >g.
\end{equation} 
Plugging the coordinates of $v_{t_0}$ into \eqref{eq:ineq-hyper} and simplify all the terms, one can obtain a quadratic inequality of $t_0$ and one can easily see that \eqref{eq:ineq-hyper}  holds if     \[
t_0 < \frac{(2-\tilde{c})(\tilde{c}^2 g -r\tilde{s})}{ \tilde{c}[r\tilde{s}-r^2-s\tilde{s}+(2-\tilde{c})g(c+\tilde{c})
] + r^2+s\tilde{s} -(c+2) r\tilde{s}}.
\]
Using the upper bound of $t_0$  given in \eqref{eq:bound}, we are reduced to check 
 \[
\begin{split}
      \tilde{c}(r\tilde{s}-r^2-s\tilde{s}+(2-\tilde{c})g(c+\tilde{c})
) + r^2+s\tilde{s} -(c+2) r\tilde{s} < (2-\tilde{c})(s\tilde{s}+r^2-2r\tilde{s}+2\tilde{c}g(\tilde{c}+c)).
\end{split}
\]
After further simplification and reduction, the inequality above becomes
\begin{equation}
    \begin{aligned}
        0 <s\tilde s+r^2+2\tilde c g+c-(\tilde c+c-2)(\tilde c^2 g-r\tilde s)
    \end{aligned}
\end{equation}
The right hand side can be estimated as below
\begin{equation}
    \begin{aligned}
   \mathrm{RHS}&>    r^2+ s\tilde s+2\tilde c g-(\tilde c+c-2)(\tilde c^2 g-r\tilde s)\\ &=  r^2+   s\tilde s +2\tilde c g+(mr+2-2c) (\tilde c^2+\frac{v^2}{2})\\
       & >  r^2+s\tilde s +2\tilde c g-r (\tilde c^2+r)\\
        &\geq (\frac{(g-1)c^2}{r}-1)(\frac{(g-1)\tilde{c}^2}{r}-1) +2\tilde c g-r\tilde c^2 \\ &>0.
    \end{aligned}
\end{equation}
\end{proof}

\section{Proof of the main theorems}\label{sec:proofofthm}

Let us first prove  Theorem \ref{mainthm}.  As the case of $m=1$ is already known, we may always assume $m>1$. There will be two cases:
\begin{enumerate}
    \item If $(g,m)\neq (7,2)$,   we can choose the Mukai vector $v=(g-1, k,k^2)$ with $k$ given in the Table  below: 

\begin{table}[ht]
\newcommand{\tabincell}[2]{\begin{tabular}{@{}#1@{}}#2\end{tabular}}
\renewcommand\arraystretch{1.5} 
\centering
\begin{tabular}{|c|c|c|}
\hline
Values  of $g$   &  Values of $k$          &  Range of $m$   \\ \hline
$3$              &  $k=5 $   &  $m\geq 5$ \\ \hline
$4$              &  $k=5$     &  $m\geq 4$ \\ \hline
$5$              &  $k=3$    &  $m\geq 3$   \\ \hline
$6$              &  $k=4$      &  $m\geq 3$    \\ \hline
$7$              & $k=5$      &  $m\geq 3$     \\ \hline
$\geq 8$         &   $k=\min\{k_0~ |~k_0>\frac{g}{2}, \gcd(g-1,k_0)=1\} $   &  $m\geq 2$     \\ \hline
\end{tabular}
\caption{Choices of Mukai vectors}
\label{tab_mukai vect}
\end{table}
\noindent Note that when $g\geq 8$, we have $$\min\left\{k_0~ |~k_0>\frac{g}{2},\gcd(g-1,k_0)=1\right\} < g-2.$$ By a direct computation, one can easily see that the values of $k$ and $m$ in Table \ref{tab_mukai vect} satisfy the  conditions in Theorem \ref{thm_surj2}  and Theorem \ref{new-region-K3}. 

\item If $(g,m)=(7,2)$, Theorem \ref{thm_surj2} fails for all primitive Mukai vectors of the form $(6,k,k^2)$.  However, we can choose $v=(2,1,3)$ and the assertion can be concluded  by the following result.

\end{enumerate}

\begin{proposition}
Suppose $g=7$.  The restriction map $\psi:\bM(2,1,3)\to \BN_C(2,1,3)$ is an isomorphism for any irreducible curve $C\in |2H|$. 
\end{proposition}
\begin{proof}
Note that $v$ satisfies the hypothesis in Corollary \ref{cor:inj num}, we know that $\psi$ is an injective morphism with stable image.   Due to Theorem~\ref{new-region-K3}, it suffices to show that $\psi$ is surjective. Suppose $\bP_v\neq \bP_{i_\ast F}$ for some $F\in \BN_C(v)$.  
A direct computation shows $\bP_v$ is at least $2$-sharp. Then either $\bP_{i_*F}$ lies inside the polygon $\bP_{0 z_1^{+2} z'_2 z_2}$, or it has $z_1^{+1}$ as a vertex. For the first case, one has
\begin{align*}
\hbar-  \|i_\ast F\| 
    &>  \|z_1- z'_1\|-\|z'_1- z_1^{+2}\|+ \|z_1- z'_2\|-\|z'_2- z_1^{+2}\| \\
    &=\frac{\sqrt{877}}{3}-\frac{\sqrt{613}}{3} \\&> \sqrt{29}+\sqrt{877}-34 ={\hbar+ \chi} - 10 
\end{align*}
which contradicts to \eqref{eq:surjective}. For the second case, it forces $\overline Z(\widetilde{E}_1)=z_1^{+1}$ and hence
\begin{equation*}
    5 \leq h^0(C, F)= h^0(X,i_*F) \leq \lfloor \frac{\ell( z_1^{+1})}{2} \rfloor+ \frac{\|z_1^{+1}- z'_2\|+\|z'_2 -z_2\|+\chi}{2} \leq \frac{\hbar+\chi}{2}.
\end{equation*}
However, we have 
\begin{equation*}
\begin{aligned}
    \frac{\hbar-(\|z_1^{+1} -z'_2\|+\|z'_2- z_2\|)}{2}-\lfloor \frac{\ell( z_1^{+1})}{2} \rfloor&= \frac{1}{6} \left(\sqrt{877}-4 \sqrt{46}\right)+\frac{\sqrt{7}}{2}-1 \\&>  \frac{1}{2} \left(\sqrt{29}+\sqrt{877}-34\right)\\&= \frac{\hbar+\chi}{2}-5
\end{aligned}
 \end{equation*}
which is impossible.  It follows from Theorem \ref{thm_polygons} that  $\psi$ is surjective. 
\end{proof}

Now we prove  Theorem \ref{mainthm2}. As in the proof of Theorem \ref{mainthm},  for each $n>0$,  we need to find Mukai vectors $v$ with $v^2=2n$ satisfying the assumptions in Corollary \ref{cor:surj2}. A key tool is  
\begin{lemma}\label{bounded-prime}
For each $n$, there is an integer $N=N(n)$ such that for $g>N$, one can find a prime number $p$ satisfying that 
\begin{enumerate}
    \item $n+1<p<\frac{\sqrt{g-1}}{2}$
    \item the equation $x^2\equiv(g-1)n \mod p$ has a solution.  
\end{enumerate}
\end{lemma}

\begin{proof}
The idea is to use the bound for prime character nonresidues. 
In \cite[Theorem 1.4]{Po17}, it has been proved  that there exists an integer $m_0$ with the property: if $j > j_0$ and $\chi$ is a quadratic character modulo $j$, there are at least $\log(j)$ primes $\ell \leq \sqrt[3]{j}$ with $\chi(\ell)= 1$. Choose $N$ to be the minimal integer satisfying \begin{itemize}
    \item $8(N-1)n\geq  j_0$,
    \item the $\lfloor\log (8(N-1)n)\rfloor$-th prime number $>n+1$,
    \item $\sqrt[3]{8(N-1)n}\leq \frac{\sqrt{N-1}}{2}$.
\end{itemize} Clearly, it only depends on $n$.  For $g>N$, we write $$g(n-1)=a^2\prod_{i=1}^k q_i ,$$
where  $q_i$ are distinct primes.  
Let $\chi_i$ be the character defined by 
$$ \chi_i(d)= \left( \frac{d }{q_i}\right)(-1)^{\frac{(d-1)(q_i-1)}{4}} $$
if $q_i$ is odd and $\chi_i(d)=(-1)^{\frac{d^2-1}{8}}$ if $q_i=2$.  Consider the quadratic character
$$\chi(d)=\prod_{i=1}^k \chi_i(d) $$
modulo $8(g-1)n$. As $8(g-1)n>N\geq j_0$, there exists a prime $p$ such that $\chi(p)=1$ and $$n+1<p<\sqrt[3]{8(g-1)n}\leq\frac{ \sqrt{g-1}}{2}.$$ Then we have its Jacobi symbol is 
$$\left( \frac{g(n-1)}{p}\right)=\prod_{i=1}^k \left( \frac{q_i }{p}\right)=\prod_{i=1}^k \chi_i(p) =\chi(p)=1$$
by the law of reciprocity. It follows that  $x^2=g(n-1)\mod p$ has a solution. 
\end{proof}

Due to Lemma \ref{bounded-prime}, when $g>N(n)$, we can find an odd prime $p$ and an integer $0<c<p$ satisfying
\begin{center}
    $n+1<p<\frac{\sqrt{g-1}}{2}$ and $p$ divides $c^2(g-1)-n$.
\end{center}
Choose the Mukai vector $v=(p,c, \frac{c^2(g-1)-n}{p})$, it automatically satisfies  all assumptions in Corollary \ref{cor:surj2}. The assertion then follows immediately.

\bibliographystyle{plain}
\bibliography{main}
\end{document}